\documentclass[12pt]{amsart}
\usepackage{amsfonts,amsmath,amssymb,amsthm,mathtools}
\usepackage{hyperref}
\usepackage{cleveref}
\usepackage{tikz-cd}
\usepackage[top=1in, bottom=1in, left=1.1in, right=1.1in]{geometry}

\newtheorem{theorem}{Theorem}[section]
\newtheorem{lemma}[theorem]{Lemma}
\newtheorem{proposition}[theorem]{Proposition}
\newtheorem{corollary}[theorem]{Corollary}

\theoremstyle{definition}
\newtheorem{definition}[theorem]{Definition}
\newtheorem{question}[theorem]{Question}
\newtheorem{example}[theorem]{Example}
\newtheorem{settings}[theorem]{Settings}

\theoremstyle{remark}
\newtheorem{remark}[theorem]{Remark}

\newcommand{\Supp}{\textup{Supp}}
\newcommand{\Spec}{\textup{Spec}}

\newcommand{\Min}{\textup{Min}}
\newcommand{\Frac}{\textup{Frac}}
\newcommand{\car}{\textup{char}}
\newcommand{\rank}{\textup{rank}}
\newcommand{\Tor}{\textup{Tor}}
\newcommand{\ann}{\textup{ann}}
\newcommand{\Hom}{\textup{Hom}}

\newcommand{\Z}{\mathbb{Z}}
\newcommand{\Q}{\mathbb{Q}}

\begin{document}
\title{Finiteness and infiniteness of gradings of Noetherian rings}
\author{Cheng Meng}
\address{Yau Mathematical Sciences Center, Tsinghua University, Beijing 100084, China.}
\email{cheng319000@tsinghua.edu.cn}
\date{\today}
\begin{abstract}
In this paper we show that for a torsion-free abelian group $G$, $\operatorname{rank}_\mathbb{Z}G<\infty$ if and only if there exists a Noetherian $G$-graded ring $R$ such that the set $\{R_g \neq 0\}$ generates the group $G$.  
For every $G$ of finite rank, we construct a $G$-graded ring $R$ such that $R_g \neq 0$ for all $g \in G$. We prove such rings give examples of PIDs which are not ED. We also use the relations between the graded division ring and the group cohomology to prove some vanishing and nonvanishing results for second group cohomology. Finally, we prove that the Hilbert series of a finitely generated $G$-graded $R$-module is well-defined when $R_0$ is Artinian, and this Hilbert series times some Laurent polynomial is equal to a Laurent polynomial. 
\end{abstract}

\maketitle
\section{Introduction}
Throughout this paper, all rings are commutative with an identity unless otherwise stated.

Let $R$ be a commutative Noetherian ring and $(G,+)$ an abelian group. We say $R$ is $G$-graded if there is a family of additive subgroups $R_g$ such that $R=\oplus_{g \in G}R_g$ and $R_gR_h \subset R_{g+h}$ for any $g,h \in G$. For a $G$-graded ring $R$, an $R$-module $M$ is $G$-graded if there is a family of additive subgroups $M_g$ such that $M=\oplus_{g \in G}M_g$ and $R_gM_h \subset M_{g+h}$ for any $g,h \in G$. We define $\Supp(R)=\{g \in G: R_g \neq 0\}$ and $\Z\Supp(R)$ to be the subgroup of $G$ generated by $\Supp(R)$. Any $G$-graded ring $R$ can be viewed as a $\Z\Supp(R)$-graded ring by ignoring other degrees.

In algebraic geometry, the coordinate rings of algebraic sets in projective spaces are $\mathbb{Z}$-graded rings, or more precisely, $\mathbb{N}$-graded rings. This is a particular kind of graded rings that usually attract one's attention. The irreducible decomposition of an algebraic set corresponds to the graded primary decomposition of $0$ ideal in its coordinate ring. In \cite{MR1727221BourbakiCommutative}, it has been proved that the primary decomposition in the graded sense is still a primary decomposition in the nongraded sense. However, we can turn to a finer decomposition which is the irreducible decomposition, and the equivalence of irreducibility and graded irreducibility has not been explored until it is proved in \cite{MR3582834Zgradedirreducible}. More generally, the author proved that when $G$ is torsion-free abelian, then $G$-graded irreducibility is equivalent to irreducibility and being $G$-graded in \cite{MR4068915Ggradedirreducible}.

In \cite{MR4068915Ggradedirreducible}, the author raises the following question:
\begin{question}\label{1.1}
Let $R$ be a Noetherian ring, $G$ be a torsion-free abelian group. Suppose $R$ is $G$-graded with $\Z\Supp(R)=G$. Does this imply that $G$ is finitely generated?    
\end{question}
If the answer to the question is yes, then in the Noetherian setting, if $R$ is $G$-graded and $G$ is torsion-free, we may always assume $G=\mathbb{Z}^n$. The finite generation property of $G$ is a strong condition. It may lead to the simplification of the proof of properties of $G$-graded rings because in this case we can prove this by induction on the number of basis elements of $G$. Also in this case, $R$ is a finitely generated $R_0$-algebra by \cite{MR706507Z[1/P]/Zgraded}, so many good properties like excellence and essential of finite type over another ring pass from $R_0$ to $R$.

In \cite{MR706507Z[1/P]/Zgraded}, Goto and Yamagishi found a ring graded over the group $G=\mathbb{Z}[1/p]/\mathbb{Z}$ which is a field. Thus if we drop the torsion-free assumption, there is a counterexample to \Cref{1.1}.
\begin{example}
Let $k$ be a field of characteristic $p>0$ and $T$ be an indeterminate. Then $k(T)$ is a purely transcendental extension of $k$. Set $K=k(T^{1/p^\infty})=\cup_n k(T^{1/p^n})$ as a subfield of an algebraic closure of $k(T)$. Then $K$ is $G=\mathbb{Z}[1/p]/\mathbb{Z}$-graded with $\deg(T^{a/p^n})=a/p^n+\mathbb{Z}$.     
\end{example}
Here $K$ is a field, so in particular, it is a Noetherian ring. It is $G$-graded with $K_g \neq 0$ for any $g \in G$. More generally, we can replace $T$ with indeterminates of any cardinality to assume $G$ has any infinite cardinality.

In this paper, we will seek for a counterexample where $G$ is torsion-free. Here is the first main result of this paper.
\begin{theorem}\label{1.3 main theorem}
Let $G$ be a torsion-free abelian group.
\begin{enumerate}
\item (See \Cref{2.7}) If $R$ is a Noetherian $G$-graded ring, then $\rank_\mathbb{Z} \Z\Supp(R)<\infty$.
\item (See \Cref{4.8}) If $\rank_\mathbb{Z} G<\infty$, then there is a Noetherian $G$-graded ring $R$ with $\Supp(R)=G$. In particular, $\Z\Supp(R)=G$.
\end{enumerate}
\end{theorem}

We give a basic example in \Cref{4.1 main example} of a $\mathbb{Q}^d$-graded Noetherian ring supported on all of $\mathbb{Q}^d$. The construction is the direct limit of $\mathbb{Z}^d$-graded rings and we show Noetherianity by proving stability condition on prime ideals. Then we use this basic example to construct examples in other senarios, and explore some properties of these rings. For one application of such rings, we prove that when $d=1$, the construction yields a PID which is not ED.
\begin{theorem}[See \Cref{4.10}]
There is a $\mathbb{Q}$-graded Noetherian ring which is a PID, but not ED.    
\end{theorem}

Next, we focus on $G$-graded fields $F$, that is, graded rings with only $0$ and unit homogeneous ideals. We relate the isomorphism classes of graded fields that are supported on $G$ and satisfy $F_0=k$ for a field $k$ with a subgroup $H^2_S(G,k^*)$ of the group cohomology $H^2(G,k^*)$. We realize this subgroup as an $S_2$-invariant subgroup induced by an $S_2$-action on $H^2(G,k^*)$. We use this relation to get some vanishing results on $H^2$ and $H^2_S$. Under this relation, we call the classes corresponding to Noetherian graded fields \textit{Noetherian classes}. Then \Cref{1.3 main theorem} gives that $H^2_S(G,k^*)$ contains some Noetherian class for some $k$. We explore the structure of the set of Noetherian classes. Here is the main result on its structure:
\begin{theorem}[See \Cref{6.13}]
Assume $G$ is infinitely generated and $H^2_S(G,k^*)$ contains a Noetherian classes. Then:
\begin{enumerate}
\item $H^2_S(G,k^*)/\Tor H^2_S(G,k^*)$ has $\mathbb{Q}$-vector space structure.
\item The set of Noetherian classes consists of cosets modulo $\Tor H^2_S(G,k^*)$.
\item The image of the set of Noetherian classes in $H^2_S(G,k^*)/\Tor H^2_S(G,k^*)$ consists of lines without the origin.
\end{enumerate}
\end{theorem}

Finally, we turn to the Hilbert function of finitely generated graded modules over Noetherian $G$-graded rings, where $G$ has finite rank, but may be infinitely generated. The Hilbert function of a graded module $M$ makes sense only when $l_{R_0}(M_g)<\infty$ for any $g$; such module is called \textit{modest}, and we give a characterization of modest modules. When the module is modest, we prove, as in the classical setting, the Hilbert-function is summable, which is a notion for nonpositive gradings corresponding to rational series for positive gradings. The main result of this part is the following:
\begin{theorem}
Let $R$ be a $G$-graded ring and $M$ a $G$-graded module. Then:
\begin{enumerate}
\item (See \Cref{7.7}) $M$ is modest if and only if $(R/\ann M)_0$ is Artinian;
\item (See \Cref{7.10}) When $M$ is modest, the Hilbert function of $M$ is summable.
\end{enumerate}
\end{theorem}

The organization of the paper is as follows. In section 2, we recall some basic notions of graded rings and modules, and prove part (1) of \Cref{1.3 main theorem}. In section 3, we introduce lemmas that will be used for constructions in section 4. In section 4 we construct a Noetherian graded ring for groups $G$ satisfying (2) of \Cref{1.3 main theorem}, thus proving (2) of \Cref{1.3 main theorem}. We also prove that under additional assumptions, the ring constructed is an example of a PID which is not ED. In section 5, we recall the notion of group cohomology and relate the second group cohomology with isomorphism classes of graded division rings. We prove that graded fields corresponds to a subgroup of the group cohomology which is the invariant group of an $S_2$-action. We prove some vanishing results on the group cohomology and relate the group addition and multiplication with Segre product and Veronese subrings. In section 6, we define the Noetherian classes inside the second group cohomology, which corresponds to Noetherian graded fields. We use this concept to get some nonvanishing results of second group cohomology, and explore the structure of the set of Noetherian classes. In section 7, we give a criterion on the modestness of a graded module, which leads to the existence of Hilbert series, and prove that modest module has summable series in $G$-graded setting.

In the rest part of the paper, we assume $(G,+)$ is a torsion-free abelian group and the unit of $G$ is $0$, unless otherwise stated.

\section{Finite rank of support}
In this section, we assume $R$ is a Noetherian $G$-graded ring. We first introduce some notations on rings graded over a general abelian group; they can be found in standard references on graded rings including \cite{MR519194Zngradedrings}, \cite{MR676974gradedringtheory} and \cite{MR2046303methodsofgradedrings}.

\begin{definition}
Let $R$ be a Noetherian $G$-graded ring, $M$ be a $G$-graded $R$-module.
\begin{enumerate}
\item Assume $H$ is another abelian group, $\pi:G \to H$ is a homomorphism of abelian groups. Then we can view $R$ as an $H$-graded ring with $R_h=\oplus_{\pi(g)=h} R_g$. We denote this graded ring structure by $R^\pi$.
\item Let $G' \subset G$ be a subgroup of $G$. Denote $R_{G'}=\oplus_{g \in G'}R_g$, which is a $G'$-graded subring of $R$. 
\item For $g \in G$, $M[g]$ is the $R$-module $M$ regraded by $M[g]_h=M_{g+h}$. 
\item A graded free module is a module of the form $\oplus_{i \in I}R[g_i]$. 
\item Let $R$ be a $G$-graded domain. Let $S$ be the set of nonzero homogeneous elements. Then the ring $S^{-1}R$ is called the homogeneous fraction field, denoted by $\Frac^h(R)$.
\item We say $R$ is a $G$-graded field if $R$ only has two homogeneous ideals, namely $0$ and $R$.
\end{enumerate}    
\end{definition}
From the definition we see if $H$ is an abelian group and $\pi:G \to H$ is a homomorphism of abelian groups, then $R^\pi_0=R_{\ker \pi}$.

Next we introduce propositions on structures of graded fields. In fact, in the cited refereces the following propositions are proved in the noncommutative setting; a $G$-graded ring which is not necessarily commutative, but has only $0$ and the unit ideal as the homogeneous ideal is called graded division ring, so a graded field is a commutative graded division ring, and propositions for graded division ring pass to graded fields.
\begin{proposition}[\cite{MR2046303methodsofgradedrings}, 2.7.1 and 4.6.1]
Let $F$ be a $G$-graded field, $G'=\Supp(F)$. Then:
\begin{enumerate}
\item $F_0$ is a field, and $G'$ is a subgroup of $G$.
\item For any $g \in G'$, $F_{g}$ is a one-dimensional $F_0$ vector space. Therefore, there is a graded $F_0$-vector space decomposition $F=\oplus_{g \in G'}F_0x_g$ where $0 \neq x_g \in F_g$, and $\{x_g,g \in G'\}$ is an $F_0$-basis of $F$.
\item Any $G$-graded module over $F$ is graded free, hence is free as ungraded module.
\end{enumerate}
\end{proposition}
\begin{proposition}[\cite{MR676974gradedringtheory}, I.4.1 and I.4.2]
Let $F$ be a $G$-graded ring.
\begin{enumerate}
\item $F$ is a $G$-graded field if and only if every nonzero homogeneous element is invertible in $F$.
\item If $F$ is a $G$-graded field and $G'$ is a subgroup of $G$, then $F_{G'}$ is also a $G$-graded field.
\end{enumerate}
\end{proposition}
\begin{proposition}\label{2.4}
Let $F$ be a $G$-graded field, $G'=\Supp(F)$.
\begin{enumerate}
\item Assume $G'' \subset G'$ is a finite free abelian group with a basis $e_1,\ldots,e_n$. Choose nonzero elements $x_i \in F_{e_i}$, then as $G$-graded rings, $F_{G''}=k[x_1,\ldots,x_n,x_1^{-1},\ldots,x_n^{-1}]$ is a Laurent polynomial ring. 
\item If $G'$ is finite free with a basis $e_1,\ldots,e_n$ and $0 \neq x_i \in F_{e_i}$, then as $G$-graded rings, $F=k[x_1,\ldots,x_n,x_1^{-1},\ldots,x_n^{-1}]$ is a Laurent polynomial ring. 
\item If $G$ is torsion-free, then $F$ is a domain.
\end{enumerate}
\end{proposition}
\begin{proof}
(1) and (2) are already proved in Theorem 3.5 of \cite{MR4068915Ggradedirreducible}, so we only prove (3). Take any two nonzero elements $a,b \in F$. Then there are finitely many degrees where the components of $a$ and $b$ are nonzero, and these degrees generate a finitely generated abelian subgroup $G''$ of $G$, but $G$ is torsion-free, so $G''$ must be a finite free abelian group. That is, $a,b \in F_{G''}$ for such $G''$. And $F_{G''}$ is a Laurent polynomial ring over a field, so it is a domain, so $ab \neq 0$. This means $F$ is a domain.
\end{proof}
\begin{proposition}
Let $R$ be a $G$-graded domain and $F=\Frac^h(R)$. Then $F$ is a $G$-graded field, and $\Supp(F)=\Z\Supp(R)$.    
\end{proposition}
\begin{proof}
Since $\deg(ab)=\deg(a)+\deg(b)$ whenever $ab \neq 0$, $R$ being a domain implies that $\Supp(R)$ is a semigroup. Now 
$F$ is graded by setting $\deg(a/b)=\deg(a)-\deg(b)$, so we have $\Supp(F)=\Supp(R)-\Supp(R)=\mathbb{Z}\Supp(R)$.    
\end{proof}
\begin{proposition}[Graded prime filtration]
Let $M$ be a finitely generated graded $R$-module. Then, there exists a finite filtration of $M$ such that each factor of the filtration is of the form $R/P[g]$ for some homogeneous prime ideal $P$ and $g \in G$.  
\end{proposition}
\begin{proof}
By results in \cite{MR1727221BourbakiCommutative} we see graded modules have homogeneous associated primes. Thus, every nonzero graded module contains a graded submodule of the form $R/P[g]$. Now we can use Noetherian induction to claim every submodule of $M$ has a finite graded prime filtration.    
\end{proof}

\begin{theorem}\label{2.7}
For any Noetherian $G$-graded ring $R$, $\rank_\Z \Z\Supp(R)<\infty$.        
\end{theorem}
\begin{proof}
We first prove this theorem when $R$ is a graded-field. In this case, $R_0$ is a field and $\Supp(R)$ is a group. We may assume $G=\Supp(R)$ by replacing $G$ with $\Supp(R)$. For any $g \in G$, $R_g$ is a one-dimensional $R_0$-vector space. We assume otherwise that $\rank_\Z G=\infty$. Then we can choose infinitely many elements $e_1,e_2,\ldots,$ such that $e_i$'s are $\mathbb{Q}$-linearly independent and $R_{e_i} \neq 0$ for $i \in \mathbb{N}$. We choose $0 \neq x_i \in R_{e_i}$. Let $y_i=x_{2i-1}+x_{2i}$. The ideal generated by all $y_i$'s is finitely generated; thus there is a finite $n$ such that $y_{n+1}$ is generated by $y_1,\ldots,y_n$. We consider $\iota: G \to \mathbb{Q}\otimes_\Z G$; this map is injective since $G$ is torsion-free. So $\iota(e_i)$ are still $\mathbb{Q}$-linearly independent. Thus they form part of $\Q$-basis of $\Q \otimes_\Z G$. We write $\Q\otimes_\Z G=\oplus_{1 \leq i \leq 2n} \Q e_i\oplus H$, then we can take a projection $\phi$ which maps $e_1,\ldots,e_{2n}$ to $0$ and preserves the other direct summands. We let $\pi=\phi\iota:G \to H$. $\pi$ maps $e_1,\ldots,e_{2n}$ to $0$ and $\pi(e_{2n+1}) \neq \pi(e_{2n+2})$. Now we consider the ring $R^\pi$ which is $R$ viewed as an $H$-graded ring via regrading, then $y_1,\ldots,y_n$ are homogeneous elements in $R^\pi$ and generate $y_{n+1}$. Therefore, they generate every homogeneous component of $y_{n+1}$. Note that $x_{2n+1}$ has degree $\pi(e_{2n+1})$, $x_{2n+2}$ has degree $\pi(e_{2n+2})$ and $\pi(e_{2n+1})\neq\pi(e_{2n+2})$, so $x_{2n+1}$ and $x_{2n+2}$ are two homogeneous components of $y_{n+1}$ in different degrees, so $y_1,\ldots,y_n$ generate $x_{2n+1}$. But $R$ is a graded-field, so $x_{2n+1}$ generates the unit ideal $R$. Thus $y_1,\ldots,y_n$ generate the unit ideal $R$. Let 
$G'=\oplus_{1 \leq i \leq 2n}\Z e_i$ and $R'=R_0[x_1,\ldots,x_{2n},x_1^{-1},\ldots,x_{2n}^{-1}] \subset R$. Then by \Cref{2.4}, $R'=R_{G'}$ is a Laurent polynomial ring in $x_1,\ldots,x_{2n}$. So $R'$ is a $G$-graded field. We view $R$ as an $R'$-module, then it is free. Therefore, the ring map $R' \to R$ is cyclic pure. So $(y_1,\ldots,y_n)R'=(y_1,\ldots,y_n)R\cap R'=R'$. But $(y_1,\ldots,y_n)R'=(x_1+x_2,\ldots,x_{2n-1}+x_{2n})$ lies in the kernel of the surjective ring map $R' \xrightarrow[]{x_{2i-1}\to (-1) , x_{2i}\to 1}R_0$, so it cannot be the unit ideal. This is a contradiction. So $\rank_\Z \Supp(R)<\infty$ for a Noetherian $G$-graded field $R$.

Now we prove that the graded field case leads to the general case. Let $R$ be a Noetherian $G$-graded ring, then $R$ has a finite prime filtration whose factors are of the form $R/P_i[g_i]$, where $P_i$'s are homogeneous prime ideals. We consider $F_i=\Frac^h(R/P_i)$, then they are Noetherian $G$-graded fields. By the previous part of the proof, $\rank_\Z\Supp(F_i)<\infty$. Since we have
$$\Supp(R) \subset \cup_i (\Supp(R/P_i)+g_i) \subset \cup_i (\Supp(F_i)+g_i) \subset \sum_i (\Supp(F_i)+\Z g_i),$$
we see $\Supp(R)$ lies inside an abelian group of finite rank, so $\rank_\Z \Supp(R)<\infty$. 
\end{proof}

\section{Lemmas for the construction}
In this section, we introduce all the lemmas for the construction of a $G$-graded Noetherian ring for any abelian group $G$ of finite $\mathbb{Z}$-rank in section 4. First, recall the following well-known lemma by Cohen:
\begin{lemma}[\cite{MR33276Cohentheorem}]
Let $R$ be a commutative ring. Then $R$ is Noetherian if and only if all the prime ideals of $R$ are finitely generated. 
\end{lemma}
It is well-known that integral extension satisfies lying over property and morphisms between spectra of rings induced by an integral extension have $0$-dimensional fibre. So we get:
\begin{corollary}
Let $\phi:R \to S$ be an integral extension. If $P \in \Spec(R)$ with $PS \in \Spec(S)$, then $PS\cap R=P$ and $PS$ is the only prime over $P$.    
\end{corollary}
\begin{lemma}\label{3.3}
Let $R$ be a commutative ring, $\{R_i, i \in \mathbb{N}\}$ be a collection of subrings of $R$ such that $R_i \subset R_{i+1}$ is an integral extension for any $i$, and $R=\cup_i R_i$. Assume for some $n$, $P_n$ is a prime ideal of $R_n$ such that $P_nR_i$ is prime in $R_i$ for any $i \geq n$. Then there is a unique prime of $R$ lying over $P_n$, which is $P_nR$.
\end{lemma}
\begin{proof}
We see $R_i \to R_j$ is integral for any $i\leq j$. Then $P_nR_i \cap R_i=P_n$ for any $i \geq n$ by last lemma. So $P_nR$ is a direct limit of primes $P_nR_i$ lying over $P_n$, hence it is a prime ideal of $R$ lying over $P_n$. But $R_n \to R$ is also integral, so $P_nR$ is the only prime lying over $P_n$.
\end{proof}

The following lemma deals with the prime ideal under purely transcendental base change.
\begin{lemma}\label{3.4}
Let $k$ be a field, $\{\alpha_i\}_{i \in I}$ be a set of indeterminates, $K=k(\alpha_i,i \in I)$,  and $R$ be a $k$-algebra.
\begin{enumerate}
\item If $R$ is a $k$-domain, then $K\otimes_k R$ is a $K$-domain.
\item For $P \in \Spec(R)$, $K\otimes_kP=P(K\otimes_kR)$ is a prime ideal of $K\otimes_kR$.
\item Let $R'=K\otimes_k R$, $I$ be an ideal of $R$, $I'=IR'$. Then for any $P' \in \Min(I')$, $P'=(P'\cap R)R'$.
\end{enumerate}
\end{lemma}
\begin{proof}
\begin{enumerate}
\item Let $A=k[\alpha_i,i\in I]$, then $A \otimes_k R$ is a domain because it is a polynomial ring over $R$, and $K\otimes_kR$ is a localization of $A\otimes_kR$. 
\item Apply (1) to $K\otimes_k(R/P)=K\otimes_kR/K\otimes_kP$.
\item By previous lemma, $(P'\cap R)R'$ is a prime ideal. Since $I \subset I' \subset P'$, $I \subset P'\cap R
$, $I'=IR' \subset (P'\cap R)R'$, so by definition of minimal primes, $P'=(P'\cap R)R'$. 
\end{enumerate}    
\end{proof}
Since $k$ is arbitrary, even if the ideal contains some indeterminates, we can still get rid of the rest indeterminates. To be precise, we have:
\begin{lemma}\label{3.5}
Let $k$ be a field, $I_1,I_2$ be two disjoint sets of indices, $K=k(\alpha_i,i \in I_1 \cup I_2)$. Let $I$ be an ideal of $K\otimes_k R$ generated by elements in $k(\alpha_i,i \in I_1)\otimes_k R$ and $P \in \Min(I)$, then $P$ can be generated by elements in $k(\alpha_i,i\in I_1)\otimes_k R$.  
\end{lemma}
\begin{proof}
Observe that $k(\alpha_i,i\in I_1)(\alpha_i,i\in I_2)=k(\alpha_i,i\in I_1\cup I_2)$ and $k(\alpha_i,i\in I_1\cup I_2)\otimes_{k(\alpha_i,i\in I_1)} k(\alpha_i,i\in I_1)\otimes_k R=k(\alpha_i,i\in I_1\cup I_2)\otimes_k R$ and apply Lemma 3.4.    
\end{proof}
We recall the following definitions on ring homomorphisms.
\begin{definition}
Let $\varphi: S_1\to S_2$ be a ring homomorphism. We say $\varphi$ is
\begin{enumerate}
\item free, if $S_2$ is a free $S_1$-module.
\item split, if $\varphi$ is injective and $S_2\cong\varphi(S_1)\oplus M$ as $S_1$-modules for some $S_1$-module $M$.
\item pure, if $M \otimes \varphi:M \to M \otimes_{S_1}S_2$ is injective for any $S_1$-module $M$.
\item cyclic pure, if $IS_2 \cap S_1=I$ for any $S_1$-ideal $I$.
\end{enumerate}
\end{definition}
It is well-known that free implies faithfully flat and split, faithfully flat or split implies pure and pure implies cyclic pure. Here are two examples of cyclic pure extensions used in our later proof.
\begin{lemma}\label{3.7cyclicpurelemma1}
\begin{enumerate}
\item Let $R$ be a $G$-graded ring, $G' \subset G$ be a subgroup, then $R_{G'} \to R$ splits.
\item Let $R$ be a $k$-algebra, $L$ be a field extension of $k$, then $R \to R\otimes_k L$ is free.
\end{enumerate}    
\end{lemma}
\begin{proof}
\begin{enumerate}
\item Let $\Lambda$ be a set of representatives of $G/G'$, then $R=\oplus_{h \in \Lambda}\oplus_{g \in G'}R_{g+h}$. We see for each $h \in \Lambda$, $\oplus_{g \in G'}R_{g+h}$ is an $R_{G'}$-module, and when $h=0$, $\oplus_{g \in G'}R_{g+h}=R_{G'}$.
\item $k \to L$ is free, so after a base change it is still free.
\end{enumerate}    
\end{proof}
We prove that for a cyclic pure map, Noetherianity passes from target to source:
\begin{lemma}\label{3.8cyclicpurelemma2}
Let $R \to S$ be a cyclic pure map. If $S$ is Noetherian, then $R$ is Noetherian. 
\end{lemma}
\begin{proof}
For any ascending chain of ideals $\{I_n\}$ of $R$, its extension $\{I_nS\}$ is also an ascending chain of ideals of $S$, so stabilizes. But $R \to S$ being cyclic pure implies $I_nS\cap R=I_n$ for any $n$, so $\{I_n\}$ also stabilizes. So $R$ is Noetherian.   
\end{proof}
\begin{lemma}\label{3.9}
Let $R$ be a domain, $K=\Frac(R)$ be its fraction field. Let $t$ be an indeterminate, $f(t) \in R[t]$ is a monic polynomial with coefficients in $R$ which is irreducible in $K[t]$. Then $S=R[t]/(f(t))$ is a domain. It is finite free over $R$; in particular, it is integral over $R$. 
\end{lemma}
\begin{proof}
Since $f(t)$ is monic, $S$ is finite free over $R$. So, $S$ embeds into $K\otimes_R S=K[t]/(f(t))$ which is a field, hence $S$ is a domain.    
\end{proof}
\begin{remark}
The converse of the above lemma also holds; if $f(t) \in R[t]$ is monic and $S=R[t]/(f(t))$ is a domain, then $f(t)$ is irreducible in $K[t]$. However, not every irreducible polynomial in $K[t]$ is a constant multiple of a monic polynomial in $R[t]$; for example, consider $R=k[x^2,x^3]$ and $f(t)=t-x \in \operatorname{Frac}(R)[t]$.
\end{remark}
\begin{lemma}\label{3.11}
Let $K$ be a field and $p \in \mathbb{N}$ be a prime number. Assume $\car(k)=p$ or $\car(k)=0$. Let $\alpha \in K$, and assume all $p$-th roots of unity in an algebraic closure of $K$ lie in $K$. Then either the polynomial $t^p-\alpha$ has a root in $K$, or it is irreducible in $K$.
\end{lemma}
\begin{proof}
Let $\beta$ be a root of $t^p-\alpha$ in an algebraic closure of $K$. 

If $\car (k)=p$, then $t^p-\alpha=(t-\beta)^p$. Any nontrivial factor of this polynomial is of the form $(t-\beta)^i$ for $1 \leq i \leq p-1$. Thus if $t^p-\alpha$ is reducible in $K$, then for some $1 \leq i \leq p-1$, $(t-\beta)^i \in K[t]$. The coefficient of $t^{i-1}$ is $i\beta$, so $i\beta \in K$. But $0 \neq i \in K$, so $\beta \in K$, so $t^p-\alpha$ has a root in $K$. 

If $\car(k)=0$, let $\xi \in K$ be a $p$-th root of unity which is not $1$. Then $\xi^i,0\leq i \leq p-1$ is pairwise distinct and $t^p-\alpha=\Pi_{0 \leq i \leq p-1}(t-\xi^i\beta)$. $\xi$ satisfies a polynomial $(t^p-1)/(t-1)$ which is irreducible over $\mathbb{Q}$. Therefore, if there is $A \subset \{0,1,\ldots,p-1\}$ such that $\sum_{i\in A}\xi^i=0$, then either $A=\{0,1,\ldots,p-1\}$ or $A=\emptyset$. If $t^p-\alpha$ is irreducible in $K$, then there is a set $A \neq \{0,1,\ldots,p-1\},\emptyset$ such that $\Pi_{i\in A}(t-\xi^i\beta) \in K(t)$. The coefficient of $t^{i-1}$ is $\sum_{i\in A}\xi^i\beta$ where $\sum_{i\in A}\xi^i$ is a nonzero element in $K$, so $\beta \in K$.   
\end{proof}
\begin{lemma}
There is a sequence of prime numbers $p_i,i \in \mathbb{N}$ such that $\mathbb{Q}=\underset{}{\varinjlim}_i \frac{1}{\Pi_{1 \leq j \leq i}p_j}\mathbb{Z}$.    
\end{lemma}
\begin{proof}
We construct a sequence of prime numbers as follows: first, put all the prime numbers in a column in the increasing order. Then we create $\mathbb{N}$-many copies of the column, creating a $\mathbb{N}*\mathbb{N}$-table:
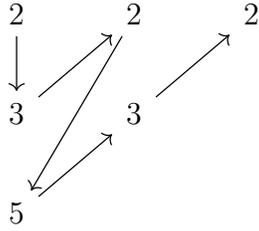
\[\begin{tikzcd}
	2 & 2 & 2 \\
	3 & 3 \\
	5
	\arrow[from=1-1, to=2-1]
	\arrow[from=2-1, to=1-2]
	\arrow[from=1-2, to=3-1]
	\arrow[from=3-1, to=2-2]
	\arrow[from=2-2, to=1-3]
\end{tikzcd}\]
Then we rearrange all the numbers one to one corresponding to $\mathbb{N}$ according to the order given by the arrow above. We run through all of them anti-diagonally: we first run through the number on the first anti-diagonal, then the numbers on the second anti-diagonal upwards, and so on. This gives us a sequence running through all prime numbers infinitely times. The first several terms of the sequence is
$p_1=2,p_2=3,p_3=2,p_4=5,p_5=3,p_6=2,p_7=7,\ldots$. As abelian groups, $\mathbb{Q}=\underset{}{\varinjlim}_i \frac{1}{\Pi_{1 \leq j \leq i}p_j}\mathbb{Z}$.    
\end{proof}

\section{Constructions of Noetherian $G$-graded rings}
Now we are ready to construct our examples of Noetherian $G$-graded rings. Let $k$ be a field of characteristic $0$. We fix a natural number $1 \leq d \in \mathbb{N}$. Let $\alpha_{1,i},\ldots,\alpha_{d,i}$ be indeterminates over $k$, $K=k(\alpha_{j,i},1\leq j \leq d,i \in \mathbb{N})$, $R_i=K[t_{1,i},\ldots,t_{d,i},t_{1,i}^{-1},\ldots,t_{d,i}^{-1}]$ be a Laurant polynomial ring of $d$-variables over $K$, and $\{p_i\}$ a sequence of prime numbers $p_i,i \in \mathbb{N}$ such that $\mathbb{Q}=\underset{}{\varinjlim}_i \frac{1}{\Pi_{1 \leq j \leq i}p_j}\mathbb{Z}$ for all $i \in \mathbb{N}$. Let $\phi_i:R_i \to R_{i+1}$ which maps $t_{j,i}$ to $\alpha_{j,i+1}^{-1}t_{j,i+1}^{p_i}$ and keeps elements in $K$. 
\begin{theorem}\label{4.1 main example}
Assume $k$ contains all roots of unity. Let $R=\underset{}{\varinjlim} R_i$ through the map $\phi_i$. Then $R$ is a $\mathbb{Q}^d$-graded Noetherian field. 
\end{theorem}
\begin{proof}
We see $R_i$ is Noetherian for any $i$. By Cohen's lemma, to prove $R$ is Noetherian, it suffices to show that for any $P \in \operatorname{Spec}(R)$, $P=P_nR$ for some $P_n \in \Spec(R_n)$. Let $P$ be a prime ideal of $R$ and $P_i=P\cap R_i,i\in \mathbb{N}$. Note that $R_1$ is Noetherian, so $P_1$ is finitely generated by Laurent polynomials in $t_{1,1},\ldots,t_{d,1}$ with finitely many coefficients in $K$. By checking these coefficients, there exists $n$ such that all the coefficients of these generators lie in $K'=k(\alpha_{j,i},1\leq j \leq d,1\leq i \leq n)$. We claim $P_nR_i$ is prime for any $i \geq n$. We see all the maps $\phi_i$ are finite free extensions, so they are integral. Therefore, the map $R_1 \to R_n$ is integral. Now $P_n$ is a prime ideal in $R_n$ with $P_1=P_n\cap R$, so $P_n \in \Min(P_1R_n)$. Let $R'_n=K'[t_{1,n},\ldots,t_{d,n},t_{1,n}^{-1},\ldots,t_{d,n}^{-1}]$. Since $P_1$ is generated by Laurent polynomials of $t_{1,1},\ldots,t_{d,1}$ in coefficients in $K'$ and the morphisms $\phi_i,1\leq i \leq n-1$ map polynomials with coefficients in $K'$ still to polynomials with coefficients in $K'$, $P_1R_n$ is generated by Laurent polynomials in $t_{1,n},\ldots,t_{d,n}$ with coefficients in $K'$, that is, generated by elements in $R'_n$. Thus by \Cref{3.5}, $P_n$ is generated by elements in $R'_n$. Now we claim $P_nR_i$ is prime for any $i \geq n+1$. By induction we only need to prove the case $i=n+1$, because if $P_n$ is generated by Laurent polynomial in $t_{1,n},\ldots,t_{d,n}$ with coefficients in $K'=k(\alpha_{j,i},1\leq j \leq d,1\leq i \leq n)$, then $P_nR_{n+1}$ is generated by Laurent polynomial in $t_{1,n+1},\ldots,t_{d,n+1}$ with coefficients in $K'=k(\alpha_{j,i},1\leq j \leq d+1,1\leq i \leq n+1)$, so we can use induction to prove $P_nR_i=P_nR_{n+1}R_i$ is prime for $i \geq n+2$. We let $R_{c,n}=\phi_n(R_n)[t_{1,n+1},\ldots,t_{c,n+1}]$ for $1 \leq c \leq d$ and set $R_{0,n}=R_n$. Note that this implies $t_{1,n+1}^{-1},\ldots,t_{c,n+1}^{-1} \in R_{c,n}$. Then $R_{n+1}=R_{d,n}$ and we extend the ring homomorphism $R_n \to R_{n+1}$ to a finer chain of rings with finite free ring homomorphisms:
$R_{0,n}=R_n \to R_{1,n} \to R_{2,n} \to \ldots \to R_{d,n}=R_{n+1}$. It suffices to prove $P_nR_{c,n}$ is prime in $R_{c,n}$ for any $0 \leq c \leq d$. The case $c=0$ is true by assumption. We assume $P_nR_{c,n}$ is a prime ideal in $R_{c,n}$, and prove $P_nR_{c+1,n}$ is a prime ideal in $R_{c+1,n}$.
$R_{c,n}=K[t_{1,n},\ldots,t_{d,n},t_{1,n}^{-1},\ldots,t_{d,n}^{-1},t_{1,n+1},\ldots,t_{c,n+1}]$ is a Laurent polynomial ring over $K$, $R_{c+1,n}=K[t_{1,n},\ldots,t_{d,n},t_{1,n}^{-1},\ldots,t_{d,n}^{-1},t_{1,n+1},\ldots,t_{c+1,n+1}]$ is a Laurent polynomial ring over $K$ which contains $R_{c,n}$ as a subring. As $R_{c,n}$-algebra, $R_{c+1,n+1}=R_{c,n}[T]/(T^{p_i}-\alpha_{c+1,n+1}t_{c+1,n})$. Let $K''=k(\alpha_{j,i},1\leq j \leq d,1\leq i \leq n$ or $1 \leq j \leq c,i=d+1)$, $R''_{c,n}=K''[t_{1,n},\ldots,t_{d,n},t_{1,n}^{-1},\ldots,t_{d,n}^{-1},t_{1,n+1},\ldots,t_{c,n+1}]$,  $P''=R''_{c,n}\cap P_nR_{c,n}$. Since $P_n$ is generated by Laurent polynomials with coefficients in $K'$, $P_nR_{c,n}$ is generated by Laurent polynomials with coefficients in $K''$, that is, $P_nR_{c,n}$ is generated by Laurent polynomials in $R''_{c,n}$, so $P_nR_{c,n}=P''R_{c,n}$. Now
\begin{align*}
R_{c+1,n}/P_nR_{c+1,n}\\
=(R_{c,n}/P_nR_{c,n})[T]/(T^{p_i}-\alpha_{c+1,n+1}t_{c+1,n})\\
=(K\otimes_{K''}(R''_{c,n}/P''))[T]/(T^{p_i}-\alpha_{c+1,n+1}t_{c+1,n}).    
\end{align*}
We denote $D=R''_{c,n}/P'$ and $L=\operatorname{Frac}(D)$. Then $D$ is a $K''$-domain and $K$ is a purely transcendental extension of $K''$. By \Cref{3.4}, $K\otimes_{K''}(R''_{c,n}/P')$ is a $K$-domain. It has characteristic $0$ and contains all roots of unity. We want to prove the quotient ring $(K\otimes_{K''}(R''_{c,n}/P''))[T]/(T^{p_i}-\alpha_{c+1,n+1}t_{c+1,n})$ is a domain. Let $\alpha=\alpha_{c+1,n+1}$, which is a free variable over $L$. We see
\begin{align*}
(K\otimes_{K''}(R''_{c,n}/P''))[T]/(T^{p_i}-\alpha_{c+1,n+1}t_{c+1,n})\\
=K\otimes_{K''(\alpha)}(K''(\alpha)\otimes_{K''}(R''_{c,n}/P''))[T]/(T^{p_i}-\alpha_{c+1,n+1}t_{c+1,n}).    
\end{align*}
By \Cref{3.4}, it suffices to prove
$(K''(\alpha)\otimes_{K''}(R''_{c,n}/P''))[T]/(T^{p_i}-\alpha_{c+1,n+1}t_{c+1,n})$ is a domain. We see $K''(\alpha)\otimes_{K''}(R''_{c,n}/P'')$ is a domain, so by 
\Cref{3.9}, it suffices to prove the polynomial $T^{p_i}-\alpha_{c+1,n+1}t_{c+1,n}$ does not have a root in $\operatorname{Frac}(K''(\alpha)\otimes_{K''}(R''_{c,n}/P''))=L(\alpha)$. 

Let $\alpha=\alpha_{c+1,n+1}$ and $t=t_{c+1,n}$, the polynomial is of the form $T^{p_i}-\alpha t \in L(\alpha)[T]$ for $t \in K''$. Since $t$ is invertible in all rings and their localizations, $t \neq 0$ in $L$. Suppose $0 \neq f(\alpha)/g(\alpha)\in L(\alpha)$ is a root of this polynomial, then $(f(\alpha)/g(\alpha))^{p_i}-\alpha t=0$ implies $f(\alpha)^{p_i}=t\alpha g(\alpha)^{p_i}$. Comparing the degrees of $\alpha$ on both sides, we get $p_i\deg(f)=p_i\deg(g)+1$. This contradicts the fact that $p_i$ does not divide $1$. This means the polynomial $T^{p_i}-\alpha t \in L(\alpha)[T]$ does not have a root. So $L(\alpha)[T]/(T^{p_i}-\alpha t)$ is a domain, thus $(K\otimes_{K''}(R''_{c,n}/P''))[T]/(T^{p_i}-\alpha_{c+1,n+1}t_{c+1,n})$ is a domain and $P_nR_{c+1,n}$ is a prime ideal in $R_{c+1,n}$. By induction, $P_nR_{d,n}=P_nR_{n+1}$ is a prime ideal in $R_{n,d}=R_{n+1}$ and the claim is proved, so $P_nR_i$ is a prime ideal in $R_i$ for any $i \geq n$ and by \Cref{3.3}, the unique prime of $R$ lying over $P_n$ is $P_nR$, so it is finitely generated, and by Cohen's Lemma, $R$ is Noetherian. 

Let $\{e_j,1\leq j \leq d\}$ be a $\mathbb{Q}$-basis of $\mathbb{Q}^d$. We define a grading on $R_n$ by setting $\deg(t_{j,i})=1/\Pi_{1 \leq i' \leq i}p_{i'}e_j$. Then $\phi_n$'s are homogeneous maps which make $R$ a graded ring graded by the group $\underset{}{\varinjlim}_{i}\oplus_{1 \leq j \leq d}1/\Pi_{1 \leq i' \leq i}p_{i'}e_j=\mathbb{Q}^d$. We see every homogeneous element of $R$ is invertible, so $R$ is a graded field.
\end{proof}
We will denote the ring constructed in \Cref{4.1 main example} by $R(k,d)$ which depends on two parameter $d \in \mathbb{N}$ and a field $k$.
\begin{theorem}
If $R(k,d)$ is Noetherian, then $R(k,d)$ is a regular ring of dimension $d$.   
\end{theorem}
\begin{proof}
It is regular because it is a limit of regular rings, and it has dimension $d$ because it is a integral extension of $R_1$ whose dimension is $d$.
\end{proof}
\begin{remark}    
We can give a concrete representation of $R(k,d)$. Let $R_1=K[t_{1,1},\ldots,t_{1,d}]$. Then inside an algebraic closure of $\Frac (R_1)$, $R(k,d)$ is the $R_1$-algebra generated by the following elements:
$$\alpha_{j,2}^{1/p_1}t_{j,1}^{1/p_1},\alpha_{j,3}^{1/p_2}\alpha_{j,2}^{1/p_1p_2}t_{j,1}^{1/p_1p_2},\alpha_{j,4}^{1/p_3}\alpha_{j,3}^{1/p_2p_3}\alpha_{j,2}^{1/p_1p_2p_3}t_{j,1}^{1/p_1p_2p_3},\ldots,1\leq j \leq d.$$
\end{remark}
\begin{corollary}\label{4.4}
Let $k$ be any field of characteristic $0$, then $R(k,d)$ is Noetherian.
\end{corollary}
\begin{proof}
Let $L$ be an extension of $k$ containing all roots of unity. Since the constructions in Theorem 4.1 only involve the indeterminates $\alpha_i$ and do not involve any element in $k$, $R(L,d)=R(k,d)\otimes_k L$ is Noetherian, hence $R(k,d)$ is Noetherian by \Cref{3.7cyclicpurelemma1} and \Cref{3.8cyclicpurelemma2}. 
\end{proof}
\begin{example}
We can also construct similar examples in prime characteristic. Take a field $k$ with $\car(k)=p$. Then if in the above example we replace all $p_i$ by $p$, we get a Noetherian ring graded over $\mathbb{Z}[1/p]^d$ where $\deg(t_{j,i})=(1/p^i)e_j$.    
\end{example}
\begin{remark}
The existence of a Noetherian ring graded over $\Z[1/q]$ of characteristic $p$ for $q \neq p$ is more complicated. In order to apply the techniques in \Cref{4.1 main example}, we need a corresponding version of \Cref{3.11}, which asks for the fact that the minimal polynomial of a $q$-th root of unity over $\mathbb{F}_p$ has degree $q-1$. Equivalently, $\min\{r \in \mathbb{N}:q|p^r-1\}=q-1$, which means $\bar{p} \in \mathbb{F}^*_q$ is a primitive element. Otherwise, it is possible to find a proper subset $A \subset \{0,1,\ldots,p-1\}$ where $\sum_{i \in A}\xi^i=0$, so \Cref{3.11} fails. For example, when $p=2,q=7$, $\xi$ is a $7$-th root of unity over $\overline{\mathbb{F}}_2$ which lies in $\mathbb{F}_8$ and satisfies $\xi^3+\xi^2+1=0$.   
\end{remark}
\begin{example}
In \Cref{4.1 main example}, $R(k,d)$ is a graded field, so $R(k,d)_0=K$ is a field. We can easily construct an example of graded ring $R$ where $R_0$ is not a field by forgetting some degrees: we see $R=R(k,d+d')$ is $\mathbb{Q}^{d+d'}$-graded. Consider the projection $\pi: \mathbb{Q}^{d+d'} \to \mathbb{Q}^d$ which ignores the last $d'$ components. Then $S=R^\pi$ is a $\mathbb{Q}^d$-graded ring with $S_0=R_{\ker \pi}\cong R(k,d')$ which is a Noetherian ring of dimension $d'$.
\end{example}
\begin{corollary}\label{4.8}
Let $k$ be a field of characteristic $0$ and $G$ be a torsion-free abelian group of finite rank. Then there is a $G$-graded field $R$ with $\Supp(R)=G$ and $R_0=k(\alpha_i,i \in \mathbb{N})$.    
\end{corollary}
\begin{proof}
Suppose $\rank_\mathbb{Z}G=d$, then we can realize $G$ as a subgroup of $\mathbb{Q}^d$. We see $R(k,d)_G$ is Noetherian by \Cref{3.7cyclicpurelemma1} and \Cref{3.8cyclicpurelemma2}. We have $\Supp(R(k,d)_G)=G$ because $R(k,d)_g \neq 0$ for any $g \in \mathbb{Q}^d$ and $(R(k,d)_G)_0=k(\alpha_{i,j})=k(\alpha_i)$ after reindexing $\alpha_i$'s.  
\end{proof}

In the case $d=1$ and $k$ is uncountable, we prove that $R(k,1)$ is a PID which is not ED. By constructions in \Cref{4.1 main example}, we set $K=k(\alpha_i,i \in \mathbb{N})$, $R_i=K[t_i,t^{-1}_i]$, $\phi_i:R_i\to R_{i+1},t_i \to \alpha_i^{-1}t_{i+1}^{p_i}$, and $R(k,1)=\varinjlim R_i$. Denote $G_i=1/p_1p_2\ldots p_i\mathbb{Z}$, and $R_{G_i}=R_i$. We have $R_{i+1}=R_i[T]/(T^{p_i}-\alpha_it_i)$.

We recall the well-known universal side divisor criterion for non-ED-ness:
\begin{proposition}
Let $R$ be a ED. Then there is a nonzero prime element $p \in R$ such that the natural projection $\pi:R \to R/pR$ satisfies $\pi(R^*)=(R/pR)^*$, where $^*$ denotes the group of units.    
\end{proposition}
\begin{proposition}\label{4.10}
If $k$ is uncountable, then $R=R(k,1)$ is a PID, but not ED.    
\end{proposition}
\begin{proof}
A domain is PID if and only if every prime ideal is principal (for instance, see Exercise 8.2.6 of \cite{MR2286236DummitandFoote}). So it suffices to prove every prime ideal of $R$ is principal. But every prime ideal is extended from some $R_i$ and $R_i$ is principal, so $R$ is principal.   

By the universal side divisor criterion, to prove $R$ is not ED, it suffices to prove for any nonzero prime element $f \in R$, $R^* \to (R/fR)^*$ is not a surjection. 

First, we may assume $f \in K[t_1,t_1^{-1}]$, $f$ is not a linear function times a unit, and all the coefficients of $f$ are in $k$. In fact, $fR\cap R_1$ is a principal prime ideal; denote $fR\cap R_1=f_0R_1$. Then $f_0$ is an irreducible Laurent polynomial in $t_1$ with finitely many coefficients in $k(\alpha_i,i \in \mathbb{N})$, so there is an $r \in \mathbb{N}$ such that the coefficients of $f_0$ lie in $k(\alpha_i,1\leq i \leq r)$. Thus, we may look at $R_{r+1}$, then the image of $f_0$ in $R_{r+1}$ is a Laurent polynomial ring with coefficients in $k(\alpha_i,1\leq i \leq r)$. It may become reducible in $R_{r+1}$, but its factors still have coefficients in the field $k(\alpha_i,1\leq i \leq r)$, so $fR\cap R_{r+1}=f_rR_{r+1}$ where $f_r$ is an irreducible Laurent polynomial in $t_1$ with coefficients in $k(\alpha_i,1\leq i \leq r)$. There are two possibilities for $f_r$: if $f_r$ is a polynomial of degree at least $2$ times a power of $t_r$, then we stop here. If $f_r$ is a polynomial of degree $1$ times a power of $t_r$, say $f_r=(t_r-c)t_r^n$. Then we consider $f_{r+1}=(\alpha_{r+1}t^{p_{r+1}}_{r+1}-c)t_{r+1}^{np_{r+1}}$ which is the image of $f_r$ in $R_{r+2}$, then it is a polynomial of degree $p_{r+1}\geq 2$ times a power of $t_{r+1}$. In the first case $f_rR=fR$, so we can replace $f$ by $f_r$, $k$ by $k(\alpha_i,1\leq i \leq r)$, $t_i$ by $t_{i+r}$, $\alpha_i$ by $\alpha_{i+r}$. In the second case we can replace $f$ by $f_{r+1}$, $k$ by $k(\alpha_i,1\leq i \leq r+1)$, $t_i$ by $t_{i+r+1}$, $\alpha_i$ by $\alpha_{i+r+1}$. In either case, we reduce to the case where the coefficients of $f$ lie in $k$ and the irreducible part of $f$ has degree $\geq 2$.

Now observe that $R$ is a $\mathbb{Q}$-graded field, so $R^*$ consists of nonzero homogeneous elements of $R$, and any nonzero homogeneous element is a product of an element in $K^*$ and a power of some $t_i's$. In particular, $R_i^*=R^*\cap R_i=K^*\times t_i^{\mathbb{Z}}$.

In the definition of $R(k,1)$ we set $\deg(t_1)=1/p_1$. For simplicity, we regrade $R(k,1)$ again by multiplying $p_1$ on all degrees; in this case $\deg(t_1)=1$ and $R_1=K[t_1,t^{-1}_1]=R(k,1)_{\mathbb{Z}}$. Note that there is a natural surjection $\pi:\mathbb{Q} \to \mathbb{Q}/\mathbb{Z}$ of abelian groups. We can regrade $R$ via $\pi$ to get a $\mathbb{Q}/\mathbb{Z}$-graded ring $R^\pi$. Since $f \in k[t_1,t_1^{-1}]$, $f$ is homogeneous in $R^\pi$ of degree $0+\mathbb{Z}$, so $R/fR$ is a $\mathbb{Q}/\mathbb{Z}$-graded ring with $(R/fR)_0=R_1/fR_1=K[t_1,t_1^{-1}]/fK[t_1,t_1^{-1}]$. Now take $x \in (R_1/fR_1)^*$ and assume it lifts to $y \in R^*$. Then $y$ is also a homogeneous element in $R^\pi$, so $\deg(x)=\deg(y)=0+\mathbb{Z}$, which means $y \in R_{G_1}=R_1$. Thus there is $c \in K^*, n \in \mathbb{N}$ such that $y=ct_1^n$.

Now we write $L=K[t_1,t_1^{-1}]/fK[t_1,t_1^{-1}]$, we see it is a finite extension of $f$ of degree $n \geq 2$; this is because the irreducible part of $f$ has degree at least $2$. If $y \in R^*$ maps to $L^*$, then $y \in R_1^*$, thus it suffices to prove $R_1^* \to L^*$ is not surjective. Under this map,
$$\textup{Im}(R_1^*)=\cup_{i \in \mathbb{Z}}K^*\overline{t_1^{i}} \subset \cup_{i \in \mathbb{Z}}K\overline{t_1^{i}}.$$
And we have
$$\cup_{i \in \mathbb{Z}}K\overline{t_1^{i}} \subsetneq L$$
since $|\mathbb{P}_K^n|=|K|$ is uncountable, so $K^n$ is not a countable union of lines in it. So the map $R_1^* \to L$ is not surjective. But elements in $R_1^*$ cannot be mapped to $0$, so $R_1^* \to L^*$ is not surjective, and we are done.

\end{proof}

\section{Graded field and $S_2$-invariant group cohomology}
Let $k$ be a field and $F$ be a graded field supported on $G$ with $F_0=k$. Then we can choose $0 \neq x_g \in F_g$ such that $F=\oplus_{g \in G} kx_g$ as a $k$-vector space. Thus, the ring structure of $F$ is determined by the multiplicative structure on $x_g$'s. For $g,h \in G$, $0 \neq x_gx_h \in F_{g+h}$, thus there is a $c_{g,h} \in k^*$ such that $x_gx_h=c_{g,h}x_{g+h}$. It turns out that we can associate the collection of $k^*$-elements, $(c_{g,h})_{(g,h)\in G\times G}$, with a class in a group cohomology.

Before we discuss the graded fields, we would like to point out that a graded ring over a field whose component is a field has been studied in more general settings in various references. In commutative algebra, we typically encounter $G$-graded fields $F$ over a field $k$ where $G$ is an abelian group. There are several kinds of generalizations in the noncommutative setting:
\begin{enumerate}
\item We do not assume $F$ is commutative. That is, if $x_gx_h$ and $x_hx_g$ live in the same degree, then their coefficients $c_{g,h},c_{h,g}$ may not be the same.
\item We do not assume $G$ is abelian; in this case, $x_gx_h$ and $x_hx_g$ may live in different degrees.
\item We do not assume $k$ is in the center of $F$. That is, $cx_g\neq x_gc$.
\item We do not assume $k$ is commutative; $k$ may be a division ring.
\end{enumerate}
To deal with rings in settings (1) and (2), we only need $c_{g,h}$ to describe the ring structure. To deal with rings in setting (3), we need to introduce an extra parameter $\sigma: G \to \operatorname{Aut}(k)$ such that commuting $cx_g$ introduces a $\sigma(g)$-action on the coefficient $c$. It is usually convenient to rule out (4), since in this case the coefficient $k^*$ of a group cohomology is not an abelian group, bringing extra obstacles for calculation.

In the following discussion we only allow generalizations of type (1).
\begin{settings}
In sections 5 and 6, we always assume $G$ is abelian, $F$ is a $G$-graded division ring with $\operatorname{Supp}(F)=G$ such that $F_0$ is a commutative ring that lies in the center of $F$. When $F_0=k$, we say $F$ is a $G$-graded division ring over $k$. Moreover, if $F$ is commutative, then it is a $G$-graded field over $k$.  
\end{settings}

We recall the following definition and propositions of group cohomology. For basic knowledge on the group cohomology, see Chapter 6 of \cite{MR1269324weibelhomological}. In these definitions we use $+$ for group action since $G$ is assumed to be abelian.
\begin{definition}
Let $G$ be a group, $A$ be an abelian group where $G$ acts on $A$ additively. View $A$ as a $\mathbb{Z}G$-module via this action, and view $\mathbb{Z}$ as a $\mathbb{Z}G$-module where $G$ acts on $\mathbb{Z}$ trivially. Then the $i$-th cohomology group of $G$ with coefficients in $A$ is $H^i(G,A):=\operatorname{Ext}^i_{\Z G}(\Z,A)$.    
\end{definition}

\begin{proposition}[Bar resolution]
The trivial $\Z G$ module $\Z$ has the following resolution $B_\bullet=(B_n,d)$, called the (unnormalized) bar resolution:
\begin{enumerate}
\item $B_n$ is the free module over symbol $[g_1\otimes\ldots\otimes g_n]$.
\item The map $\epsilon:B_0 \to \Z$ maps $[]$ to $1$.
\item The map $d_i: B_n \to B_{n-1}$, $0 \leq i \leq n$ is given by
\begin{enumerate}
\item $d_0[g_1\otimes\ldots\otimes g_n]=g_1[g_2\otimes\ldots\otimes g_n]$.
\item $d_i[g_1\otimes\ldots\otimes g_n]=[g_1\otimes \ldots\otimes (g_i+g_{i+1})\otimes\ldots\otimes g_n]$ for $1 \leq i \leq n-1$.
\item $d_n[g_1\otimes\ldots\otimes g_n]=[g_1\otimes\ldots\otimes g_{n-1}]$.
\end{enumerate}
\item $d=\sum_{0 \leq i \leq n}(-1)^id_i: B_n \to B_{n-1}$.
\end{enumerate}
\end{proposition}
\begin{remark}
As a consequence, $H^2(G,A)=H^2(\Hom_{\Z G}(B_\bullet,A))=Z^2(G,A)/B^2(G,A)$. We identify $Z^2(G,A),B^2(G,A)$ as subgroups of $\Hom_{\Z G}(B_2,A)=\Hom_{set}(G\times G,A)$, that is, the set of set-theoretic functions from $G\times G$ to $A$. Then:
\begin{enumerate}
\item $\varphi \in Z^2(G,A)$ if and only if $\varphi(1,g)=\varphi(g,1)$ and $f\cdot\varphi(g,h)-\varphi(f+g,h)+\varphi(f,g+h)-\varphi(f,g)=0$.
\item $\varphi \in B^2(G,A)$ if and only if $\varphi(1,g)=\varphi(g,1)$ and there exists $\beta:G \to A$ such that $\varphi(f,g)=f\beta(g)-\beta(f+g)+\beta(f)$.
\end{enumerate}
\end{remark}
\begin{remark}
For a possibly infinite collection of abelian groups $\{A_j\}_{j \in \Lambda}$, we see
$$\Hom_{\Z G}(B_2,\prod_j A_j)=\prod_j \Hom_{\Z G}(B_2,A_j)$$
and the conditions defining $Z^2, B^2$ can be verified componentwisely. Thus
$Z^2(G,\prod_{j \in \Lambda} A_j)=\prod_{j \in \Lambda}, Z^2(G,A_j),B^2(G,\prod_{j \in \Lambda} A_j)=\prod_{j \in \Lambda} B^2(G,A_j),H^2(G,\prod_{j \in \Lambda} A_j)=\prod_{j \in \Lambda} H^2(G,A_j)$. In particular, these functors commute with finite direct sum.    
\end{remark}
The following theorem shows that the structure of graded division ring supported on $G$, which is not necessarily commutative, corresponds to its second group cohomology:
\begin{theorem}[\cite{MR676974gradedringtheory}, I.4.7]\label{5.6}
Let $G$ be a group, $F_0$ be a field, $F=\oplus_{g \in G}kx_g$ be a $G$-graded field, $c:G\times G \to k^*$, $c(g,h)=c_{g,h}$ be a map of sets. We view $k^*$ as a $G$-module via trivial action.
\begin{enumerate}
\item Extending the map $x_g\times x_h=c_{g,h}x_{g+h}$ $k$-linearly gives $F$ a $G$-graded ring structure if and only if $c \in Z^2(G,k^*)$.
\item The two ring structures on $F$ given by $c,c' \in Z^2(G,k^*)$ are isomorphic as $G$-graded rings if and only if $c-c' \in B^2(G,k^*)$.
\item The set of isomorphism class of $G$-graded division ring $F$ supported on $G$ with $F_0=k$ is in one-to-one correspondence with the group cohomology $H^2(G,k^*)$.
\end{enumerate}
\end{theorem}
\begin{definition}
Let $G$ be a group, $k$ be a field, $\tilde{c} \in Z^2(G,k^*)$ which maps to $c \in H^2(G,k^*)$. We denote by $k[G,\tilde{c}]$ or $k[G,c]$ the graded field $F$ which satisfies $F=\oplus_{g \in G} kx_g$ and $x_gx_h=\tilde{c}_{g,h}x_{g+h}$. Under this notation, the group algebra is $k[G]=k[G,0]$, which is the $k$-vector space $k[G]=\oplus_g kx_g$ endowed with the $k$-linear multiplication $x_gx_h=x_{g+h}$. 
\end{definition}
The Noetherian property of a group algebra only depends on the finite generation property of $G$, and is independent of $k$:
\begin{proposition}[\cite{MR4068915Ggradedirreducible}, Theorem 3.7]\label{5.8}
Let $G$ be a group, $k$ be a field, then $k[G]$ is Noetherian if and only if $G$ is finitely generated.    
\end{proposition}

In general, even if $G$ is abelian, a class in $Z^2(G,k^*)$ may correspond to a noncommutative ring. For example, the ring $k<x,y,x^{-1},y^{-1}>/(xy-qyx),q \neq 0,1$ is a $\Z^2$-graded division ring which is noncommutative. This ring is a ``quantum version" of the Laurent polynomial ring. It is not isomorphic to the group algebra as the group algebra is commutative. This implies $H^2(\Z^2,k^*)\neq 0$ for any $k\neq \mathbb{F}_2$.

For a $G$-graded division ring $F$ defined by the class of a function $c \in Z^2(G,k^*)$ where $G$ is abelian, the commutation of $x_g,x_h$ leads to the equality $c_{g,h}=c_{h,g}$. Thus we may check the $S_2$-action on all functions $G\times G \to k^*$ which permutes the two lower indices.
\begin{lemma}
Let $G$ be an abelian group, $A$ be a trivial $G$-module. Then there is an $S_2$-action on functions $G\times G \to A$: if $S_2=\{e,\sigma\}$, then $\sigma \varphi(g,h)=\varphi(h,g)$, and $Z^2(G,A)$ and $B^2(G,A)$ are all $S_2$-invariant.
\end{lemma}
\begin{proof}
This is an $S_2$-action since $\sigma$ commutes with addition and subtraction of functions, and $\sigma^2\varphi(g,h)=\sigma\varphi(h,g)=\varphi(g,h)$, so $\sigma^2$ acts via identity. Now it suffices to prove that the identities defining $Z^2(G,A)$ and $B^2(G,A)$ are preserved under $\sigma$. If for all $g$, $\varphi(1,g)=\varphi(g,1)$, then $\sigma\varphi(g,1)=\sigma\varphi(1,g)$, which means $\sigma\varphi(1,g)=\sigma\varphi(g,1)$.Suppose for all 
$f,g,h$, $f\cdot\varphi(g,h)-\varphi(f+g,h)+\varphi(f,g+h)-\varphi(f,g)=0$, then $\varphi(g,h)-\varphi(f+g,h)+\varphi(f,g+h)-\varphi(f,g)=0$ because the action is trivial. Then $\sigma\varphi(h,g)-\sigma\varphi(h,f+g)+\sigma\varphi(g+h,f)-\sigma\varphi(g,f)=0$. Since $f,g,h$ are arbitrary, we can switch $f,h$ to get $\sigma\varphi(f,g)-\sigma\varphi(f,h+g)+\sigma\varphi(g+f,h)-\sigma\varphi(g,h)=0$. This implies $\sigma\varphi(f,g)-\sigma\varphi(f,g+h)+\sigma\varphi(f+g,h)-f\cdot\sigma\varphi(g,h)=0$ because $G$ is abelian and the action is trivial. If for some $\beta:G \to A$, $\varphi(f,g)=f\beta(g)-\beta(f+g)+\beta(f)$, then $\sigma\varphi(f,g)=\varphi(g,f)=g\beta(f)-\beta(g+f)+\beta(g)=\beta(f)-\beta(f+g)+f\beta(g)$ since $G$ is abelian and the action is trivial. So all the conditions defining $Z^2(G,A)$ and $B^2(G,A)$ are preserved under the $S_2$-action, so they are $S_2$-invariant.
\end{proof}
\begin{definition}
We define $Z^2_S(G,A)=Z^2(G,A)^{S_2}$ which is the $S_2$-invariant subgroup of $Z^2(G,A)$, $B^2_S(G,A)=B^2(G,A)^{S_2}$ which is the $S_2$-invariant subgroup of $B^2(G,A)$, $H^2_S(G,A)=Z^2_S(G,A)/B^2_S(G,A)$. We call $H^2_S(G,A)$ the $S_2$-invariant second group cohomology. There is an induced $S_2$-action on $H^2(G,A)$, and $H^2_S(G,A)$ is the $S_2$-invariant subgroup of $H^2(G,A)$.   
\end{definition}
\begin{theorem}\label{5.11}
Let $G$ be an abelian group, $A_i,i \in \Lambda$ be abelian groups with a trivial $G$-action. Then $H^2_S(G,\prod_{i \in \Lambda} A_i)=\prod_{i \in \Lambda} H^2_S(G,A_i)$.    
\end{theorem}
\begin{proof}
Both sides of the equation sit in $H^2(G,\prod_{i \in \Lambda} A_i)=\prod_{i \in \Lambda} H^2(G,A_i)$. We see a function $G\times G \to \prod_{i \in \Lambda}A_i$ is $S_2$-invariant if and only if every component is invariant, so the equality holds.   
\end{proof}
\begin{theorem}
For $c \in H^2(G,k^*)$, $k[G,c]$ is commutative if and only if $c \in H^2_S(G,k^*)$.    
\end{theorem}
\begin{proof}
Let $\tilde{c}$ be a lifting of $c$ to $Z^2(G,A)$. Then $F=k[G,c]$ is commutative if and only if $x_gx_h=x_hx_g$ for all $g,h \in G$. Since $x_{gh}=x_{hg}$, this means $\tilde{c}_{g,h}=\tilde{c}_{h,g}$, that is, $\tilde{c} \in Z^2_S(G,k^*)$, so $c \in H^2_S(G,k^*)$.    
\end{proof}
\begin{theorem}\label{5.13}
Let $k$ be a field, $G=\mathbb{Z}^n$, then $H^2_S(G,k^*)=0$.    
\end{theorem}
\begin{proof}
By \Cref{2.4}. 
\end{proof}
\begin{remark}
In general, we only have that $H^2_S(G,k^*)$ is a subgroup of $H^2(G,k^*)$. If $k^*$ is $2$-divisible, then it is a direct summand. They are not always the same; we see $H^2(\Z^2,k^*)\neq 0$ for $k \neq \mathbb{F}_2$ by the existence of a quantum Laurent polynomial ring, but $H^2_S(\Z^2,k^*)=0$ for any $k$ by \Cref{5.13}. 
\end{remark}
The phenomenon in the above remark only appears in rank at least $2$, as shown below:
\begin{proposition}\label{5.15}
Let $G$ be an abelian group of rank $1$. Then every $G$-graded division ring $F$ is commutative. Equivalently, $H^2(G,k^*)=H^2_S(G,k^*)$ for any field $k$.    
\end{proposition}
\begin{proof}
We can write $G$ as a union of copies of $\Z$ containing one another: $G=\cup_i \Z e_i$, $e_i \in \Z e_{i+1}$. We consider $F_{\Z e_i}$ which is a $\Z$-graded division ring over $k$. We take any $x_{e_i}\in F_{e_i}$, then we must have $F_{\Z e_i}=k[x_{e_i},x^{-1}_{e_i}]$ which is a commutative Laurent polynomial ring. So $F=F_G=\cup_i F_{e_i}$ is commutative.    
\end{proof}
\Cref{5.6} allows us to use knowledge on $G$-graded field to derive information on $S_2$-invariant group cohomologies, especially about vanishing and non-vanishing properties.
\begin{definition}
We say a field $k$ is root-closed if for any $n \in \mathbb{N}$, any element in $k$ has an $n$-th root.
\end{definition}
We see that algebraically closed fields are root-closed. On the other hand, root-closed fields are not necessarily algebraically closed; this is the famous Abel-Ruffini theorem.
\begin{lemma}\label{5.17}
Let $G$ be an abelian group of rank $d$. Then there is a filtration of $G$ by finite free abelian groups of rank $d$, $G_1 \subset G_2 \subset G_3 \subset \ldots$ satisfying the following property:
\begin{enumerate}
\item $G=\cup_i G_i=\underset{i}{\varinjlim}G_i$.
\item For any $i$, there are $d+1$-elements $e_1,\ldots,e_n,f$ such that:
\begin{enumerate}
\item $G_i=\oplus_{1 \leq j \leq d}\Z e_j$.
\item $G_{i+1}=\oplus_{1 \leq j \leq d-1}\Z e_j \oplus \Z f$.
\item $e_d=nf$ for some $n$ depending on $i$.
\end{enumerate}
\end{enumerate}
\end{lemma}
\begin{proof}
Since $G$ has rank $d$, there are $d$ elements in $G$ that are $\mathbb{Q}$-linearly independent. We denote the abelian group generated by these elements by $G_1$, then $G_1 \cong \mathbb{Z}^d$. We see $G$ embeds into $\mathbb{Q}^d$, so in particular, $G$ is countable. So we can choose countably many elements $a_2,a_3,\ldots$ such that $G$ is generated by $G_1,a_2,a_3,\ldots$. Write $G_i=G+\sum_{2 \leq j \leq i}\Z a_j$, then $G_i \cong \Z^d$, $G_i=G_{i-1}+\Z a_i$, and $G=\cup_i G_i=\underset{i}{\varinjlim}G_i$. Consider the filtration $G_1 \subset G_2 \subset G_3 \subset \ldots$, then it suffices to prove (2) for this filtration. Choose one basis $g_1,\ldots,g_d$ of $G_i$. Then as abelian groups, $G_{i+1}=G_i\oplus \Z a_i/\Z r$ where $r$ is a generator of the relations among $g_1,\ldots,g_d,a$. We write $r=r_1g_1+\ldots+r_dg_d+r_{d+1}a$. Since $G_{i+1}$ is torsion-free, $\gcd(r_1,\ldots,r_{d+1})=1$. Let $M=\gcd(r_1,\ldots,r_d)$, $r=Mg+r_{d+1}a$, $g=1/M(r_1g_1+\ldots,r_dg_d)$. Since $1/M\cdot \gcd(r_1,\ldots,r_d)=1$, we see $g$ is a basis element of $G_i$. Thus after a change of basis we may assume $g=g_d,r=Mg_d+r_{d+1}a$ where $\gcd(M,r_{d+1})=1$. So $G_i=\Z g_d \oplus \oplus_{1 \leq i \leq d-1}\Z g_i$ and $G_{i+1}=G_i+M/r_{d+1}\Z g_1$. But we see $\Z+ M/r_{d+1}\Z=1/r_{d+1}\Z$, so $G_{i+1}=1/r_{d+1}\Z g_d \oplus \oplus_{1 \leq i \leq d-1}\Z g_i$. Choose $e_1=g_1,\ldots,e_d=g_d,f=1/r_{d+1}g_d$, then they satisfy (2), so we are done.
\end{proof}
\begin{theorem}\label{5.18 root-closed vanishing}
Let $G$ be an abelian group of rank $d$ which is not necessarily finitely generated. Let $k$ be a root-closed field. Then any $G$-graded field over $k$ is isomorphic to the group algebra. Equivalently, $H^2_S(G,k^*)=0$. In particular, this is true when $k$ is algebraically closed.
\end{theorem}
\begin{proof}
If $G$ is finite free, then $G=\mathbb{Z}^d$, so any $G$-graded field is a Laurent polynomial ring and the theorem is true. Now we assume $G$ is infinitely generated and choose a filtration $G_1 \subset G_2 \subset G_3 \ldots$ of finite free abelian groups as in \Cref{5.17}. Let $F$ be a $G$-graded field over $k$. It suffices to find a system of homogeneous $k$-basis elements $0 \neq x_g \in F_g$ such that $x_gx_h=x_{gh}$ for all $g,h \in G$; we say such a homogeneous $k$-basis is good. Note that $F_{G_1}$ is a $G_1$-graded field where $G_1\cong \Z^d$, so $F_{G_1}$ has a good homogeneous $k$-basis. Suppose we have found a good homogeneous $k$-basis for $F_{G_i}$, we claim that they extend to a good homogeneous $k$-basis for $F_{G_{i+1}}$. Actually, we write $G_i=\oplus_{1 \leq i \leq d}\Z e_i$, $G_{i+1}=\oplus_{1 \leq i \leq d-1}\Z e_i \oplus \Z f$, $nf=e_d$. Let $x_{e_d}$ be the degree $e_d$ good basis element, and choose any $y \in F_f$, then $y^n=cx_{e_d}$ for some $c \in k^*$. Since $k$ is root-closed, we can find $c' \in k^*$ with $c=c'^n$. Thus let $x_f=1/c'y$, we see $(x_f)^n=x_{e_d}$. Thus we can extend the good basis from $F_{G_i}$ to $F_{G_{i+1}}$ by adjoining monomials in $x_f$ as basis elements. By induction, a good basis of $F_{G_1}$ extends to a good basis of $F_{G_i}$ for every $i$; taking union over all $i$, we get a good basis of $F$. Since a good basis exists, we see $F$ is isomorphic to the group algebra $k[G]$. Equivalently, $H^2_S(G,k^*)=0$.
\end{proof}
\begin{example}
In \Cref{4.1 main example}, $R=R(k,d)$ is a Noetherian $K$-domain. Let $\bar{K}$ be the algebraic closure of $K$. Then $\bar{K}\otimes_K R$ is a graded field over an algebraically closed field $\bar{K}$, so by \Cref{5.18 root-closed vanishing}, $\bar{K}\otimes_K R \cong \bar{K}[\mathbb{Q}^d]$. Since $\mathbb{Q}^d$ is not finitely generated, $\bar{K}\otimes_K R$ is not Noetherian. This is an example of a Noetherian algebra over a field whose base change is not Noetherian.  
\end{example}

\Cref{5.18 root-closed vanishing} leads to the following vanishing result:
\begin{corollary}
Let $G$ be an abelian group of finite rank.
\begin{enumerate}
\item $H^2_S(G,\mathbb{C}^*)=0$.
\item $H^2_S(G,\mathbb{R}_{>0})=H^2_S(G,\mathbb{R})=H^2_S(G,\mathbb{R}/\mathbb{Z})=0$.
\item $H^2_S(G,\mathbb{Q})=H^2_S(G,\mathbb{Q}/\mathbb{Z})=H^2_S(G,\mathbb{Z}[1/p]/\mathbb{Z})=0$ where $p$ is a prime number.
\item If 
$\rank G=1$, then $H^2(G,\mathbb{Q})=H^2(G,\mathbb{Q}/\mathbb{Z})=H^2(G,\mathbb{Z}[1/p]/\mathbb{Z})=0$ where $p$ is a prime number.
\end{enumerate}
\end{corollary}
\begin{proof}
(1) is a consequence of \Cref{5.18 root-closed vanishing}. Now we see the arg-length expression of a nonzero complex number gives group isomorphism $\mathbb{C}^*\cong \mathbb{R}_{>0}\oplus \mathbb{R}/\mathbb{Z}$, thus by \Cref{5.11}, $H^2_S(G,\mathbb{R}_{>0})=H^2_S(G,\mathbb{R}/\mathbb{Z})=0$. And $\mathbb{R}\cong \mathbb{R}_{>0}$ as abelian groups via the exponential map, so (2) is true. (3) is true since as abelian groups, $\mathbb{Q}$ is a direct summand of $\mathbb{R}$, $\mathbb{Q}/\Z$ is a direct summand of $\mathbb{R}/\Z$, and $\mathbb{Z}[1/p]/\mathbb{Z}$ is a direct summand of $\mathbb{Q}/\Z$. (4) is the consequence of (3) and \Cref{5.15}.
\end{proof}
Next we state two torsion properties of the group cohomology. The first theorem is trivial:
\begin{theorem}\label{5.21}
Let $G$ be an abelian group and $k$ is a finite field. Set $r=|k^*|$. Then $H^2(G,k^*)$ is $r$-torsion. Therefore, $H^2_S(G,k^*)$ is also $r$-torsion.
\end{theorem}
\begin{theorem}\label{5.22}
Let $G$ be an abelian group of finite rank, $k$ be a field and $L$ be its root-closed closure. Suppose $n=[L:k]<\infty$. Then $H^2_S(G,k^*)$ is $n$-torsion.
\end{theorem}
\begin{proof}
We observe that the composition $k^*\xhookrightarrow{}L^*\xrightarrow[]{N}k^*$ is equal to the map that raises to $n$-th power where $N$ is the norm map. So the composition of the induced maps
$$H^2(G,k^*)\xrightarrow[]{\varphi_1}H^2(G,L^*)\xrightarrow[]{\varphi_2}H^2(G,k^*)$$
is multiplication by $n$ on $H^2(G,k^*)$. Since the $G$-action on $k^*$ and $L^*$ are trivial, the maps $\varphi_1,\varphi_2$ are $S_2$-equivariant. So for $c \in H^2_S(G,k^*)$, $\varphi_1(c) \in H^2_S(G,L^*)=0$ since $L$ is root-closed. So $nc=\varphi_2\varphi_1(c)=0$.
\end{proof}

Next we show that the addition and scalar multiplication of $H^2_S(G,k^*)$ correspond to the Segre product of two rings and Veronese subrings. Note that since $G$ is abelian, the integer multiple $g \to rg$ are group homomorphisms.
\begin{theorem}
Let $G$ be a group, $k$ be a field. Let $F=k[G,c],F'=k[G,c']$ be $G$-graded fields determined by the class of $c,c' \in Z^2_S(G,k^*)$ respectively. 
\begin{enumerate}
\item Consider the natural $G\times G$-graded ring on $F\otimes_kF'$ by $\deg(x\otimes y)=(\deg(x),\deg(y))$. Let $F\sharp_kF'=(F\otimes_k F')_{\Delta(G)}$ where $\Delta(G)=\{(g,g),g \in G\}\subset G\times G$ is the diagonal subgroup.  Then $F\sharp_kF'=k[G,c+c']$.
\item For $r \in \mathbb{Z}\backslash\{0\}$, consider the group homomorphism $i:rG \to G: rx \to x$, then $F^r=(F_{rG})^i$ is a $G$-graded field. We have $F^r=k[G,rc]$. In particular, let $\pi: G \to G, g \to -g$ and $F'=F^{\pi}$, then $F'=k[G,-c]$.
\item $F\sharp_kF'$ is a direct summand of $F\otimes_kF'$ and $F^r$ is a direct summand of $F$.
\end{enumerate}
\end{theorem}
\begin{proof}
We write $F=\oplus_g kx_g$, $x_gx_h=c_{g,h}x_{g+h}$, $F'=\oplus_g ky_g$, $y_gy_h=c'_{g,h}y_{g+h}$, and let $r$ be a nonzero integer. By the commutativity of $F,F'$, $F\sharp_k F'$ has a graded vector space decomposition $F\sharp_k F'=\oplus_{g}kx_gy_g$ with $x_gy_gx_hy_h=c_{g,h}c'_{g,h}x_{g+h}y_{g+h}$. $F^r$ has a graded vector space decomposition $F^r=\oplus_g kx^r_g$ with $x^r_gx^r_h=c^r_{g,h}x^r_{g+h}$. So we are done. 
\end{proof}
\begin{remark}
We also get the following example where the tensor product of two Noetherian $k$-domains over $k$ is a non-Noetherian $k$-domain: let $G$ be an infinitely generated abelian group of finite rank, let $k$ be a field such that there is a Noetherian class $c \in H^2_S(G,k^*)$. Let $F,F'$ be the $G$-graded field determined by the class $c,-c$ respectively, then $F,F'$ are two Noetherian $k$-domains. Consider the tensor product $F\otimes_kF'$, then it is a $k$-domain. It cannot be Noetherian because it has a non-Noetherian direct summand $F\sharp_k F'\cong k[G]$.     
\end{remark}
\section{Noetherian classes in group cohomology}
We see in Section 4 there is a Noetherian $G$-graded field for every torsion-free group $G$ of finite rank. We will discuss the set of classes in second group cohomology corresponding to these graded fields.
\begin{definition}
Let $G$ be a group, $k$ be a field. We say a class $c \in H^2_S(G,k^*)$ is a Noetherian class, if $k[G,c]$ is Noetherian.    
\end{definition}
First we prove some existence and non-existence properties of Noetherian classes.
\begin{proposition}\label{6.2}
Let $G$ be a torsion-free group of finite rank, $k$ be a field.
\begin{enumerate}
\item If $G$ is finite free, then there is only one class in $H^2_S(G,k^*)$, namely $0$, which is Noetherian. If $G$ is not finitely generated, then $0$ is not a Noetherian class.
\item Let $G$ be any group of finite rank and $k=k_0(\alpha_i,i \in \mathbb{N})$ where $char(k_0)=0$. Then $H^2_S(G,k^*)$ contains at least one Noetherian class, and if $G$ is not finite free, then it cannot be the $0$ class.
\end{enumerate}    
\end{proposition}
\begin{proof}
This is a reformulation of \Cref{4.4} and \Cref{5.8}.    
\end{proof}
\begin{theorem}
Let $k$ be a field of characteristic $0$, $R(k,1)$ be as in \Cref{4.1 main example}, and $K=\Frac (R(k,1))$. Let $G$ be any abelian group of finite rank. Then $H^2_S(G,K^*)$ has a Noetherian class.   
\end{theorem}
\begin{proof}
It suffices to prove when $G=\mathbb{Q}^d$. Construct $F=R(k,d+1)$ and view it as $\mathbb{Q}^d$-graded ring by ignoring the first component, then $F_0=R(k,1)$, and $F$ is Noetherian. So $F'=F\otimes_{R(k,1)}K$ is a localization of Noetherian ring, hence is Noetherian. We see $F_g$ is a free $R(k,1)$-module of rank $1$, so $F'_g$ is a one-dimensional $K$-vector space, hence $F'$ is a $G$-graded field over $K$ and we are done. 
\end{proof}

\begin{lemma}\label{6.4}
Let $G$ be an abelian group of rank $d$ and $r$ be a positive integer. Then $|G/rG|\leq r^d$.    
\end{lemma}
\begin{proof}
We can realize $G$ as a direct limit of finitely generated abelian groups $G_i$ of rank $d$ indexed by $\mathbb{N}$. In particular, $G_i \cong \mathbb{Z}^d$. Since taking direct limit is an exact functor, we see $G/rG \cong \underset{i}{\varinjlim}G_i/rG_i=\underset{i}{\varinjlim}(\Z/r\Z)^d$ where we omit the maps between different copies of $(\Z/r\Z)^d$'s. Note that $|(\Z/r\Z)^d|=r^d$. Thus, if we pick more than $r^d$ elements in $G/rG$, then their images lie in one $(\Z/r\Z)^d$, so two of them coincide in this copy of $(\Z/r\Z)^d$ and they still coincide in the direct limit $G/rG$. Thus $G/rG$ has at most $r^d$ elements.    
\end{proof}
\begin{theorem}\label{6.5 Noetherian class multiple}
Let $G$ be an abelian group of finite rank, $k$ be a field, $F$ be a $G$-graded field over $k$, $r$ be a nonzero integer.
\begin{enumerate}
\item $F$ is a finitely generated $F^r$-algebra.
\item $F$ is Noetherian if and only if $F^r$ is Noetherian.
\item A class $c \in H^2_S(G,k^*)$ is Noetherian if and only if $rc$ is Noetherian.
\item If $G$ is infinitely generated, then any Noetherian class is torsion-free.
\item Let $c'\in H^2_S(G,k^*)$ be a torsion class, then $c$ is Noetherian if and only if $c+c'$ is Noetherian.
\end{enumerate}
\end{theorem}
\begin{proof}
We may assume $r>0$, otherwise we may replace $F$ by $F^{-1}$ which is isomorphic to $F$ as rings. Now we see $F^r=F_{rG}$ as rings. Let $\Lambda$ be a set of representatives of $G/rG$, then $|\Lambda|<\infty$ by \Cref{6.4}. Now we see $F=F_{rG}[x_g|g \in \Lambda]$, so it is finitely generated over $F^r$, and (1) is true. For (2), by finite generation, $F^r$ being Noetherian implies that $F$ is Noetherian. But the map $F^r \to F$ splits, so $F$ being Noetherian also implies that $F^r$ is Noetherian. (3) is a reformulation of (2) as $k[G,c]^r=k[G,rc]$. For (4), we see a nonzero integer multiple of a Noetherian class is also Noetherian and since $G$ is infinitely generated, $0$ is not Noetherian. So any nonzero integer multiple of a Noetherian class is nonzero, which means that a Noetherian class is torsion-free. For (5), assume $c'$ is $r$-torsion, then $c$ is Noetherian if and only if $rc=r(c+c')$ is Noetherian if and only if $c+c'$ is Noetherian.
\end{proof}
\begin{corollary}\label{6.6}
Let $G$ be an abelian group of finite rank. Assume one of the following holds:
\begin{enumerate}
\item $k$ is root-closed.
\item $k$ is finite.
\item $[L:k]<\infty$ where $L$ is the root-closure of $k$.
\end{enumerate}
Then $H^2_S(G,k^*)$ is either $0$ or $r$-torsion for some $r$. In particular, we have:
\begin{enumerate}
\item[(a)] For any $G$-graded field $F$ with $F_0=k$, $F^r=k[G]$ for some $r$.
\item[(b)] Moreover, if $G$ is not finitely generated, then $F$ is not Noetherian. So there is no Noetherian class in $H^2_S(G,k^*)$.
\end{enumerate}
\end{corollary}
\begin{proof}
This is true by \Cref{5.18 root-closed vanishing}, \Cref{5.21}, \Cref{5.22}, and \Cref{6.5 Noetherian class multiple}.  
\end{proof}

We also check the behavior of Noetherian classes when we have different groups or different base fields. The following lemmas are straightforward to check.
\begin{lemma}
Let $H \subset G$ be an extension of abelian groups and $A$ an abelian group with trivial $G$-action. Then there is a natural restriction map $\Hom_{Set}(G\times G, A) \to \Hom_{Set}(H\times H,A)$ which is $S_2$-equivariant and preserves $Z^2$, $B^2$. Thus there are induced natural maps from $Z^2(G,A)$,
$B^2(G,A)$, $H^2(G,A)$, $Z^2_S(G,A)$, $B^2_S(G,A)$, $H^2_S(G,A)$ to $Z^2(H,A)$,
$B^2(H,A)$, $H^2(H,A)$, $Z^2_S(H,A)$, $B^2_S(H,A)$, $H^2_S(H,A)$ respectively. 
\end{lemma}
\begin{lemma}
Let $c \in H^2(G,k^*)$ and denote its image by $\bar{c} \in H^2(H,k^*)$. Then $k[H,\bar{c}]=k[G,c]_{H}$.
\end{lemma}
\begin{lemma}
Let $G$ be an abelian group, $A,B$ be two abelian groups with trivial $G$-actions. Then any homomorphism of abelian groups $A \to B$ induces maps on $\mathcal{F}(A) \to \mathcal{F}(B)$ for $\mathcal{F}=Z^2(G,\cdot),B^2(G,\cdot),H^2(G,\cdot),Z^2_S(G,\cdot),B^2_S(G,\cdot),H^2_S(G,\cdot)$. 
\end{lemma}
\begin{lemma}
Let $K \subset L$ be a field extension and $G$ be an abelian group. Consider $K^*,L^*$ as trivial $G$-modules. Let $c \in H^2_S(G,K^*)$ and $\bar{c} \in H^2_S(G,L^*)$ be its image under the natural map induced by the inclusion $K^* \xhookrightarrow{} L^*$. Then $L[G,\bar{c}]=K[G,c]\otimes_K L$.
\end{lemma}
\begin{theorem}
Let $K \subset L$ be a field extension, $H \subset G$ be a group extension. Let $c \in H^2_S(G,K^*)$, $c_1 \in H^2_S(H,K^*)$ and $c_2 \in H^2_S(G,L^*)$ be the images of $c$ under the natural maps. Then:
\begin{enumerate}
\item If $c$ is Noetherian, then $c_1$ is Noetherian.
\item If $c_1$ is Noetherian and $[G:H]<\infty$, then $c$ is Noetherian.
\item If $c_2$ is Noetherian, then $c$ is Noetherian.
\item If $c$ is Noetherian and $[L:K]<\infty$, then $c$ is Noetherian.
\end{enumerate}
\end{theorem}
\begin{proof}
We make the following observation: let $F$ be a $G$-graded field over $K$, then:
\begin{enumerate}
\item $F_{H} \to F$ splits.
\item If $[G:H]<\infty$, then $F$ is a finitely generated $F_{H}$-module.
\item $F \to F\otimes_K L$ splits.
\item If $[L:K]<\infty$, then $F\otimes_K L$ is a finitely generated $F$-module.
\end{enumerate}
We see Noetherian property passes to direct summands and finitely generated algebras (so in particular, module-finite extension), so we are done.
\end{proof}

We prove that a Noetherian class on group $G$ can be extended to a Noetherian class on a larger group $G'$ when $[G':G]<\infty$.
\begin{theorem}\label{6.12}
Let $G \subset G'$ be two abelian groups of finite rank with $[G':G]<\infty$ and $F$ a $G$-graded field over a field $k$. Then $F$ extends to a $G'$-graded field over $k$.    
\end{theorem}
\begin{proof}
By considering the $\mathbb{Z}$-module prime filtration of $G'/G$, it suffices to prove the case where $G'/G=\mathbb{Z}/p\mathbb{Z}$ and $p$ is a prime number. We write $G'=G+\mathbb{Z}e$ where $pe \in G, e \notin G$. Choose any element $0 \neq y \in F_{pe}$. Since $F$ is a domain, $F$ is contained in an algebraically closed field. Let $x=y^{1/p}$ be a $p$-th root of $y$ on an algebraically closed field, then $F[x]$ is a domain and there is a surjection $F[T]/(T^p-y) \to F[x]$ sending $T$ to $x$. We claim it is an isomorphism. Equivalently, $T^p-y$ is the minimal polynomial of $x$ over $\Frac (F)$. Suppose otherwise $T^p-y$ is reducible in $\Frac (F)$, then $x$ is the root of a polynomial of degree $1\leq d \leq p-1$ in $\Frac (F)$. We may assume this $d$ is minimal. By setting $\deg(T)=e$, we see $F[T]$ is $G'$-graded. By clearing denominators and taking a nonzero homogeneous component of the leading coefficient, we see $x$ is the root of a homogeneous polynomial
$$a_dT^d+\ldots+a_0$$
where $a_d,\ldots,a_0$ are homogeneous elements in $F$, $a_d \neq 0$. Since $0 \neq x$ and $d$ is minimal, $a_0 \neq 0$. So $de=\deg(T^d)=\deg(a_0)-\deg(a_d)\in G$. Since $(d,p)=1$, then $pe \in G$ and $de \in G$ implies $e \in G$, which is a contradiction. So $T^p-y$ is the minimal polynomial of $x$ over $\Frac (F)$, and $F[x] \cong F[T]/(T^p-y)$ is a $G'$-graded field.
\end{proof}
\begin{theorem}\label{6.13}
Let $G$ be an abelian group of finite rank, $F$ be a $G$-graded field over $k$, $r$ be a positive integer. Then:
\begin{enumerate}
\item There is a $G$-graded field $F'$ over $k$ with $F'^r=F$.
\item Every class in $H^2_S(G,k^*)$ is $r$-divisible.
\item There is a $\mathbb{Q}^*$-action on $H^2_S(G,k^*)/\Tor (H^2_S(G,k^*))$, which makes it a $\mathbb{Q}$-vector space. 
\item If $G$ is infinitely generated, then the set of Noetherian classes in $H^2_S(G,k^*)$ consisits of cosets of the subgroup $\Tor (H^2_S(G,k^*))$. Moreover, the images of Noetherian classes in $H^2_S(G,k^*)/\Tor (H^2_S(G,k^*))$ form a set of lines without the vertex. 
\end{enumerate}
\end{theorem}
\begin{proof}
By \Cref{6.4} $[G:rG]<\infty$, so by \Cref{6.12} we can extend $F$ to a $G'=(1/r)G$-graded field and view it as a $G$-graded field $F'$ via $1/rG\cong G$, then $F=F'^r$. (1) implies (2) is trivial. For (3), we claim for $r \in \mathbb{N}$, there is a well-defined $1/r$-action: for $\bar{c} \in H^2_S(G,k^*)/\Tor (H^2_S(G,k^*))$, lift it back to $c \in H^2_S(G,k^*)$ and find $c' \in H^2_S(G,k^*)$ with $rc'=c$. The difference of different choices of $c'$ is $r$-torsion, so $\bar{c'} \in H^2_S(G,k^*)/\Tor (H^2_S(G,k^*))$ is unique, and we define $(1/r)\bar{c}=\bar{c'}$. This gives a $\mathbb{Q}^*$-action, and an abelian group with a $\mathbb{Q}^*$-action is just a $\mathbb{Q}$-vector space. For (4), note that Noetherian classes are stable under adding a torsion class by \Cref{6.5 Noetherian class multiple}(5). If $G$ is infinitely generated, then by \Cref{6.5 Noetherian class multiple}(3), $c$ is Noetherian if and only if $rc$ is Noetherian for some $r \in \mathbb{Q}^*$ and $rc \neq 0$ for any $r \in \mathbb{Q}^*$, so the images of Noetherian classes form a set of lines without the vertex. 
\end{proof}
We end this section with two questions. The first question is about the existence of Noetherian classes under certain base fields. For an arbitrary field $K$ of characteristic $0$, we cannot expect the existence of Noetherian classes over $K$ itself due to \Cref{6.6}, and there is a Noetherian classes over a transcendental extension of $K$ of infinite transcendence degree by \Cref{6.2}. So it is natural to ask what happens for finite transcendental extension:
\begin{question}
Let $K$ be a field of characteristic $0$. Does $H^2_S(G,K(\alpha)^*)$ always contain a Noetherian class? 
\end{question}

The second question is on the structure of the set of Noetherian classes. Given that Noetherian classes exist, we cannot expect that the set of Noetherian classes is a semigroup of the abelian group $H^2_S(G,k^*)$, and its image in $H^2_S(G,k^*)/\Tor H^2_S(G,k^*)$ is not a cone either. This is true since when $G$ is not finite free, then for any Noetherian class $c$, $-c$ is also Noetherian, but $c+(-c)=0$ which is not a Noetherian class. Now it is natural to ask:
\begin{question}
Is there any more algebraic structure (or analytic structure if $k$ has some analytic structure) on the set of Noetherian classes in $H^2_S(G,k^*)$ when it is nonempty?    
\end{question}

\section{Existence and rationality of Hilbert series}
Let $R$ be a $G$-graded Noetherian ring with $\mathbb{Z}\operatorname{Supp}R=G$. We see $\rank_{\mathbb{Z}}(G) < \infty$, so we can embed $G$ into $\mathbb{Q}^d$ for some $d$ and view $R$ as a $\mathbb{Q}^d$-graded ring. In this section, we will define the $G$-graded Hilbert function and Hilbert series of a $G$-graded module. This is a generalization of the multigraded Hilbert function and multigraded Hilbert series defined in Chapter 1 and 8 of Strumfels' book \cite{SturmfelsCA}. Note that, to define a Hilbert function or series, a necessary condition is that each degree piece has finite length. So we recall the following definition whose $\mathbb{Z}^n$-graded version appears in Definition 8.37 of \cite{SturmfelsCA}:
\begin{definition}[modest module]
Let $R$ be a $G$-graded ring and $M$ be a $G$-graded $R$-module. We say $M$ is modest if $l_{R_0}(M_g)<\infty$ for any $g$.    
\end{definition}
We will give a description showing that modest property is a mild condition on the module.
\begin{lemma}\label{7.2}
Let $R$ be a $G$-graded Noetherian domain. Assume $g \in G$ with $R_{-g} \neq 0$. Then $R_g$ is finitely generated over $R_0$.     
\end{lemma}
\begin{proof}
$R$ being Noetherian implies that $R_0$ is also Noetherian. Since $R$ is a domain, multiplication by an element of degree $g^{-1}$ induces an injection of $R_0$-modules $R_g \to R_0$. The result follows since $R_0$ is Noetherian. 
\end{proof}
\begin{definition}
We say a semigroup $N \subset \mathbb{R}^d$ is pointed if $N \cap -N =0$.    
\end{definition}
If $N$ is pointed, then by Hahn-Banach separation theorem there is a linear functional $h:\mathbb{R}^d \to \mathbb{R}$ separating $N$ and $-N$. In particular, this is true for semigroups $N \subset \mathbb{Q}^d$. Changing the sign of $h$ if necessary, we may assume $h(N) \subset \mathbb{R}_{\geq 0}$.
\begin{lemma}\label{7.4}
Let $R$ be a $G$-graded Noetherian domain. Let $N=\Supp(R)$, which is a semigroup. Assume $N$ is pointed. Then $R$ is a finitely generated $R_0$-algebra.     
\end{lemma}
\begin{proof}
First we may assume $N \subset G \subset \mathbb{Q}^d \subset \mathbb{R}^d$ for some $d$. We induct on $d$ such that $N \subset \mathbb{R}^d$. If $d=1$, then we may assume $N \subset G=\mathbb{Q} \subset\mathbb{R}$ is pointed. For such $N$ we must have $N \subset \mathbb{Q}_{\geq 0}$ or $N \subset \mathbb{R}_{\leq 0}$. We may assume $N \subset \mathbb{R}_{\geq 0}$; otherwise we may regrade elements in $R$ by replacing $R_g$ with $R_{-g}$, that is, we replace $R$ by $R^{-1}$ in the sense of Theorem 5.22(2). Now since $R$ is positively graded, $R_{>0}=(a \in R_g,g >0)$ is an $R$-ideal. It is finitely generated, say generated by $x_1,\ldots,x_s$. We claim $R=R_0[x_1,\ldots,x_s]$. Actually it suffices to show $y \in R$ is a homogeneous element then $y \in R_0[x_1,\ldots,x_s]$. If $\deg(y)=0$, then this is true. If $\deg(y)>0$, then $y \in R_{>0}$, so there are homogeneous elements $y_1,\ldots,y_s$ such that $y=y_1x_1+\ldots+y_sx_s$ where $\deg(y_1),\ldots, \deg(y_s)\leq \deg(y)-\min\{\deg(x_1),\ldots,\deg(x_s)\}$ and $\min\{\deg(x_1),\ldots,\deg(x_s)\}>0$, so by an induction on the degree we see $y \in R_0[x_1,\ldots,x_s]$. So $R=R_0[x_1,\ldots,x_s]$ and the case $c=1$ is proved. 

Now assume $d \geq 2$. We assume $N \subset\mathbb{R}^d$ is a pointed semigroup. Then by pointedness, there is a nonzero linear functional $h: \mathbb{R}^d \to \mathbb{R}$ such that $h(N) \subset \mathbb{R}_{\geq 0}$. We can view $R$ as $\mathbb{R}_{\geq 0}$-graded ring $R^h$ via $h:\mathbb{R}^d \to \mathbb{R}$. Now $R^h$ is positively graded, so $R^h$ is finitely generated $R^h_0$-algebra by $d=1$ case. So $R$ is finitely generated over $R^h_0=R_H$ where $H=h^{-1}(0) \subset \mathbb{R}^d$. We make the following observation:
\begin{enumerate}
\item $R$ being Noetherian implies that $R_H$ is Noetherian.
\item $\Supp(R_H)=H \cap N$ is pointed because $(H\cap N)\cap -(H\cap N) \subset N \cap -N \subset \{0\}$.
\item $\dim_\mathbb{R}\mathbb{R}H<d$ because $h \neq 0$.
\end{enumerate}
Therefore, the $\mathbb{R}H$-graded ring $R_H$ satisfies all the hypotheses in the induction process except that 
$\dim_{\mathbb{R}}\mathbb{R}H<d$. So by induction on $d$, $R_H$ is a finitely generated $R_0$-algebra. Thus $R$ is a finitely gnerated $R_0$-algebra.
\end{proof}
\begin{lemma}\label{7.5}
Let $R$ be a $G$-graded Noetherian domain. Let $N=\Supp(R)$, $G'=N\cap (-N)=\{g \in G|R_g \neq 0, R_{-g} \neq 0\}$. Then:
\begin{enumerate}
\item $G'=\mathbb{Q}G' \cap N$.
\item There is a homomorphism $\pi: G \to \mathbb{Q}^{c}$ for some $c$ such that $G'=\ker \pi\cap N$ and $\pi(N)$ is pointed.
\item $R$ is finitely generated over $R_{G'}$, and $R_{G'}$ is Noetherian.
\item There exists finitely many elements $g_1,\ldots,g_s$ such that $N$ is generated by $G'$ and $g_1,\ldots,g_s$ as a monoid.  
\item For any $g \in G$, $R_g$ is finitely generated over $R_0$.
\end{enumerate}  
\end{lemma}
\begin{proof}
(1) Suppose $g \in G'$ such that $a/bg \in N$. We want to show $a/bg \in G'$. By definition $R_{-g} \neq 0$, $R_{a/bg} \neq 0$, and $R$ is a domain, so $R^a_{-g}R^{b-1}_{a/bg} \neq 0$, but $R^a_{-g}R^{b-1}_{a/bg} \subset R_{-a/bg}$, so $R_{-a/bg} \neq 0$, so $-a/bg \in G'$. 

(2) Let $\pi:G \to \mathbb{Q}G \to \mathbb{Q}G/\mathbb{Q}G'$, then by (1), $G'=\ker \pi\cap N$. Suppose there is $v \neq 0$ such that $v,-v \in \pi(N)$. Then after lifting to $\mathbb{Q}G$ and multiplying an integer we find $w_1,w_2 \in N$ such that $\pi(w_1)+\pi(w_2)=0 \in G/G'$, so $w_1+w_2=g \in G'$. Then $-g \in G' \subset N$, so $-w_1=w_2-g \in N$, which implies $w_1 \in G'=\ker \pi\cap N$, which contradicts $\pi(w_1) \neq 0$.

(3) We view $R$ as a $\mathbb{Q}^c$-graded ring $R^\pi$ via $\pi: G \to \mathbb{Q}^c$, then $\{g\in \mathbb{Q}^c|R^\pi_g \neq 0\}=\pi(N)$ is pointed. By \Cref{7.4}, $R=R^\pi$ is a finitely generated $R^\pi_0=R_{\ker \pi}$-algebra, and since $R$ is supported in degrees in $N$, $R_{\ker \pi}=R_{\ker \pi\cap N}=R_{G'}$. $R_{G'}$ is Noetherian by \Cref{3.7cyclicpurelemma1} and \Cref{3.8cyclicpurelemma2}.

(4) Choose $x_1,\ldots,x_s$ such that $R=R_{G'}[x_1,\ldots,x_s]$. We may assume they are all homogeneous, otherwise we replace them by their homogenous components. Then $\{\deg(x_i),1 \leq i \leq s\}$ and $G'$ generates $N$.

(5) We assume $R=R_{G'}[x_1,\ldots,x_s]$ where $\deg(x_i)=g_i \notin G'$ and fix a choice of such $x_i,g_i$. For every $g \in N$, we claim that the following equality
$$g=g'+\sum_i a_ig_i, g' \in G',a_i \in \mathbb{N}$$
has only finitely many solutions. Every such solution is determined by a choice of $a_i$ as $g'=g-\sum_i a_ig_i$. Now we consider $\pi(g_i) \in \mathbb{Q}^c$ as in step (2), then $\pi(g_i)$ generates a pointed semigroup. Since this pointed semigroup is finitely generated, we can find a linear function $\phi: \mathbb{Q}^c \to \mathbb{Q}$ such that $\phi(\pi(g_i))>0$ for every $i$. Thus there is $\delta>0$ such that $\phi(\pi(g_i))\geq \delta$ for every $i$. Apply $\phi\pi$ to the equation above and using the fact $\pi(g')=0$ we see
$$\phi(\pi(g))=\sum_i a_i\phi(\pi(g_i))\geq \delta\sum_ia_i.$$
Thus $\sum_ia_i\leq 1/\delta\phi(\pi(g))$. Since $a_i\in\mathbb{N}$ for all $i$, there are only finitely many solutions of $a_i$.

Now from $R=R_{G'}[x_1,\ldots,x_s]$, we see $R_g=\sum_{g'+\sum_i a_ig_i=g} R_{g'}x_1^{a_1}\ldots x_n^{a_n}$. Since $R_{g'}$ is finitely generated over $R_0$ for $g'\in G'$ by \Cref{7.2} and this is a finite sum of $R_0$-modules by the previous step, $R_g$ is finitely generated over $R_0$.
\end{proof}
\begin{theorem}\label{7.6}
Let $R$ be a $G$-graded Noetherian ring and $M$ be a finitely generated $R$-module. Then $M_g$ is finitely generated over $R_0$ for any $g \in G$.    
\end{theorem}
\begin{proof}
Every graded module has a graded prime filtration of finite length; therefore, it suffices to prove the case $M=R/P[g_0]$ for some homogeneous prime $P$ and $g_0 \in G$. By \Cref{7.5}, $R/P[g_0]_g=R/P_{g+g_0}$ is finitely generated over $(R/P)_0$, so it is also finitely generated over $R_0$.
\end{proof}
Now we give a characterization of modestness. It is easy to see that for an exact sequence $0 \to M' \to M \to M'' \to 0$, $M$ is modest if and only if $M'$ and $M''$ are both modest; in particular, if $M$ has a filtration, then $M$ is modest if and only if all the factors are modest. Also shifting in degrees does not change modestness.
\begin{theorem}\label{7.7}
Let $R$ be a $G$-graded Noetherian ring and $M$ be a finitely generated $R$-module. Then the following are equivalent:
\begin{enumerate}
\item $M$ is modest.
\item $R/P$ is modest for all $P \in \Min(\ann (M))$.
\item $R/P$ is modest for all $P \in V(\ann (M))$.
\item $R/\ann (M)$ is modest.
\item $(R/\ann (M))_0$ is Artinian.
\end{enumerate}
\end{theorem}
\begin{proof}
(1) implies (2): choose a graded prime filtration of $M$. For every $P \in \Min(\ann (M))$, $R/P[g_P]$ must appear because when we ignore the grading, the factor must appear $l_{R_P}(M_P)>0$ times. So $M$ being modest implies $R/P$ is modest.

(2) implies (3): For $P \in V(\ann (M))$ there exists $Q \in \Min(\ann (M))$ with $Q \subset P$. So $R/P$ surjects onto $R/Q$ and $R/P$ being modest implies $R/Q$ is modest.

(3) implies (1): choose a graded prime filtration of $M$, then any factor has the form $R/P$ where $P \in V(\ann (M))$, so if all the $R/P$'s are modest, then $M$ is modest.

That (4) implies (2) and (3) implies (4) is a particular case of (1) implies (2) and (3) implies (1) since $\ann (R/\ann (M))=\ann (M)$.

(4) implies (5) by definition and (5) implies (4) by \Cref{7.6}. So we are done.
\end{proof}
Let $M$ be a modest $G$-graded module where $G$ is a subgroup of $\mathbb{Q}^d$. We can define the Hilbert function and Hilbert series of $M$. Let $\mathbb{Z}[\mathbb{Q}^d]$ be the set of finite sum $\sum_{t \in \mathbb{Q}^d}a_tz^t$ where $z^t=z_1^{t_1}\ldots z_d^{t_d}$. This is an $\mathbb{Z}[z_1,\ldots,z_d]$-algebra. Let $\mathbb{Z}[[\mathbb{Q}^d]]$ be the set of formal infinite sum $\sum_{t \in \mathbb{Q}^d}a_tz^t$ where $z^t=z_1^{t_1}\ldots z_d^{t_d}$. This is not a ring or a module over $\mathbb{Z}[[z_1,\ldots,z_d]]$. However, it is a module over $\mathbb{Z}[z_1,\ldots,z_d]$. This way of defining the Hilbert series inside a module $\mathbb{Z}[[\mathbb{Q}^d]]$ is a generalization of Definition 1.10 and Definition 8.14 of \cite{SturmfelsCA}. We will write $\mathbb{Z}[z]=\mathbb{Z}[z_1,\ldots,z_d]$ and $\mathbb{Z}[z,z^{-1}]=\mathbb{Z}[z_1,\ldots,z_d,z_1^{-1},\ldots,z_d^{-1}]$ for simplicity.
\begin{definition}
Let $G$ be a subgroup of $\mathbb{Q}^d$, $R$ be a $G$-graded ring, $M$ be a finitely generated $R$-module. Assume $R_0$ is Artinian. The Hilbert function of $M$ is the following function: if $t \in G$, then
$$HF(t)=l_{R_0}(R_t),$$
and if $t \in \mathbb{Q}^d\backslash G$, then $HF(t)=0.$
The Hilbert series is an element in $\mathbb{Z}[[\mathbb{Q}^d]]$ of the form $HS(z)=\sum_{t \in \mathbb{Q}^d}l_{R_0}(M_t)z^t$.
\end{definition}
\begin{definition}
We say an element $a(z) \in \mathbb{Z}[[\mathbb{Q}^d]]$ is summable, if there exists $b(z) \in \mathbb{Z}[z,z^{-1}]$ such that $a(z)b(z)\in\mathbb{Z}[\mathbb{Q}^d]$.  
\end{definition}
The notion of summable series is a generalization of rational series. The $\mathbb{Z}^n$-graded case of summability is defined in Definition 8.39 of \cite{SturmfelsCA}. Actually, if the grading of $R$ is positive, then $a(z)\in\mathbb{Z}[[\mathbb{Q}^d_+]]$ which is a domain, and $a(z)b(z)=c(z)$ implies $a(z)=b(z)/c(z)$ in its fraction field. The point of this generalization comes from the case where the grading of $R$ is not positive. In this case, $\mathbb{Z}[[\mathbb{Q}^d]]$ is not a ring and has $\mathbb{Z}[z,z^{-1}]$-torsion elements. Therefore, we can only write $a(z)\equiv b(z)/c(z)$ and $a(z)$ is not uniquely determined by $b(z),c(z)$. It is proven (\cite{SturmfelsCA}, Theorem 8.41) that if $G=\mathbb{Z}^n$, then finitely generated $G$-graded modules have summable Hilbert series, and we prove that the same thing holds for any $G \subset \mathbb{Q}^d$:
\begin{theorem}\label{7.10}
Let $R$ be a $G$-graded ring, $M$ be a modest $G$-graded module. Then there exists finitely many $v_1,\ldots,v_s \in G$ such that $HS(z)\Pi_i (1-z^{v_i})$ is a Laurant polynomial. That is, the Hilbert series of $M$ is summable.   
\end{theorem}
\begin{proof}
Since $G$ is torsion-free, we can take a graded prime filtration of finite length of $M$. Since Hilbert function is additive on short exact sequences, we may reduce to the case that $M=R/P,P \in V(\ann (M))$ is modest. We may replace $R/P$ by $R$ again to assume $M=R$ is a domain because for $R/P$-modules, the $R/P$-module length is equal to the $R$-module length. In this case, $R_0$ has finite length and is a domain, so it is a field. Denote $G'=\Supp(R)\cap (-\Supp(R))$ as in \Cref{7.5}. If $G' \neq 0$, then take $g \in G'$ and $0 \neq a \in R_g', 0 \neq b \in R_{-g'}$. Then $0 \neq ab \in R_0$, so $a$, $b$ are units. So multiplication by $a$ induces an isomorphism of $R_0$-vector spaces $R_g \to R_{g+g'}$, so $HS(z)(1-z^{g'})=0$. If $G'=0$, then $R=R_0[x_1,\ldots,x_s]$ where $x_i$'s are all homogeneous. So $R$ is graded over a finite free abelian group, and we can apply the result in $\mathbb{Z}^n$-graded case to prove $HS(z)\Pi_i (1-z^{v_i})$ is a Laurant polynomial for some $v_i \in \Supp(R)$.  
\end{proof}
\begin{remark}
At the end of \cite{SturmfelsCA}, 8.4, there is a question asking for a generalization of Hilbert series
for an immodest graded module $M$. One possible generalization for Hilbert series would be the Hilbert-Samuel function with respect to some ideal. Let $I$ be a homogeneous ideal such that $(R/I+\ann (M))_0$ is Artinian, then for any $n$, $M/I^nM$ and $I^nM/I^{n+1}M$ are modest modules. The double series
$$\sum_{n \geq 0}HS_{I^nM/I^{n+1}M}(z)w^n$$
is summable because it is just the Hilbert series of the $G\times\mathbb{Z}$-graded module $gr_I(M)$. Therefore, 
$$\sum_{n \geq 0}HS_{M/I^{n+1}M}(z)w^n=\sum_{n \geq 0}HS_{I^nM/I^{n+1}M}(z)w^n(1+w+w^2+\ldots)$$
is also summable.
\end{remark}

\section*{Acknowledgement}
This material is based upon work supported by the National Science Foundation under Grant No. DMS-1928930 and by the Alfred P. Sloan Foundation under grant G-2021-16778, while the author was in residence at the Simons Laufer Mathematical Sciences Institute (formerly MSRI) in Berkeley, California, during the Spring 2024 semester. The author would like to thank Linquan Ma for reading an early draft.
\bibliographystyle{plain}
\bibliography{refQNoe}
\end{document}